\documentclass[12pt]{article}

\usepackage{amsfonts,amsmath,amsthm,amssymb,amscd}
\usepackage[dvips]{graphicx}

\theoremstyle{definition}
\newtheorem{theorem}{Theorem}

\newtheorem{proposition}[theorem]{Proposition}

\numberwithin{equation}{section}
\numberwithin{theorem}{section}

\begin{document}

\begin{center}
{\bf{\Large On quintic identities }}
\end{center}

\begin{center}
By Kazuhide Matsuda
\end{center}

\begin{center}
Faculty of Fundamental Science, National Institute of Technology, Niihama College, \\
7-1 Yagumo-chou, Niihama, Ehime, 792-8580, Japan \\
e-mail: matsuda@sci.niihama-nct.ac.jp  
\end{center}

\noindent
{\bf Abstract}
In this paper, 
we derive quintic versions of the cubic identities of Farkas and Kra. 
In particular, we believe that our results can be easily generalized to $k$ th power versions, $(k=7,9,11,\ldots).$
Moreover, we investigate the algebraic structure of theta constants of level five. 

\section{Introduction}
Throughout this paper, 
let $\mathbb{N},$ $\mathbb{N}_0,$ $\mathbb{Z},$ $\mathbb{Q},$ $\mathbb{R}$ and $\mathbb{C}$ denote 
the set of positive integers, nonnegative integers, integers, quotient numbers, real numbers and complex numbers, 
respectively. 
Moreover, 
let the upper half plane be defined by 
$$
\mathbb{H}^2=\{ \tau \in\mathbb{C} \, | \, \Im \tau >0 \}. 
$$
\par
Following Farkas and Kra \cite{Farkas-Kra-1}, 
we introduce the theta function with characteristic 
$\left[
\begin{array}{c}
\epsilon \\
\epsilon^{\prime}
\end{array}
\right] \in\mathbb{R}^2,$ 
which is defined by  
\begin{equation*}
\theta 
\left[
\begin{array}{c}
\epsilon \\
\epsilon^{\prime}
\end{array}
\right] (\zeta, \tau) 
=
\theta 
\left[
\begin{array}{c}
\epsilon \\
\epsilon^{\prime}
\end{array}
\right] (\zeta)
:=\sum_{n\in\mathbb{Z}} \exp
\left(2\pi i\left[ \frac12\left(n+\frac{\epsilon}{2}\right)^2 \tau+\left(n+\frac{\epsilon}{2}\right)\left(\zeta+\frac{\epsilon^{\prime}}{2}\right) \right] \right),
\end{equation*}
which uniformly and absolutely converges on compact subsets of $\mathbb{C}\times\mathbb{H}^2,$ 
where $\mathbb{H}^2$ is the upper half-plane.  
The theta constants are defined by 
\begin{equation*}
\theta 
\left[
\begin{array}{c}
\epsilon \\
\epsilon^{\prime}
\end{array}
\right]
:=
\theta 
\left[
\begin{array}{c}
\epsilon \\
\epsilon^{\prime}
\end{array}
\right] (0, \tau). 
\end{equation*}
\par
Farkas and Kra \cite{Farkas-Kra-1} 
treated the theta constants with rational characteristics, 
that is, 
the case where $\epsilon$ and $\epsilon^{\prime}$ are both rational numbers, 
and 
derived a number of interesting theta constant identities. 
\par
One of the most famous theta constant identities is Jacobi's quartic, 
which is given by 
\begin{equation*}
\theta^4 
\left[
\begin{array}{c}
0 \\
0
\end{array}
\right]
=
\theta^4 
\left[
\begin{array}{c}
1\\
0
\end{array}
\right]
+
\theta^4 
\left[
\begin{array}{c}
0 \\
1
\end{array}
\right]. 
\end{equation*}
\par
Based on the uniformization theory, 
Farkas and Kra \cite{Farkas-Kra-paper} proved the following cubic identities: 
for every $\tau\in\mathbb{H}^2,$ 
\begin{equation}
\label{Farkas-Kra-cubic-(1)}
\theta^3
\left[
\begin{array}{c}
\frac13 \\
\frac13
\end{array}
\right]
+
\theta^3
\left[
\begin{array}{c}
\frac13 \\
\frac53
\end{array}
\right]
=
\theta^3
\left[
\begin{array}{c}
\frac13 \\
1
\end{array}
\right],
\end{equation}
and 
\begin{equation}
\label{Farkas-Kra-cubic-(2)}
\exp\left(\frac{\pi i}{3}\right)
\theta^3
\left[
\begin{array}{c}
\frac13 \\
\frac13
\end{array}
\right]
+
\exp\left(\frac{2\pi i}{3}\right)
\theta^3
\left[
\begin{array}{c}
\frac13 \\
\frac53
\end{array}
\right]
=
\theta^3
\left[
\begin{array}{c}
1 \\
\frac13
\end{array}
\right].
\end{equation}
\par
Our aim is to obtain quintic versions of the identities (\ref{Farkas-Kra-cubic-(1)}) and  (\ref{Farkas-Kra-cubic-(2)}). 
Moreover, we investigate the algebraic structure 
of the following theta constants:
\begin{equation*}
\theta
\left[
\begin{array}{c}
\frac{j}{5} \\
\frac{k}{5}
\end{array}
\right],
\,\,
\theta
\left[
\begin{array}{c}
1 \\
\frac{j}{5}
\end{array}
\right],  \,\,
(j=1,3, k=1,3,5,7,9). 
\end{equation*}
\par
This paper is organized as follows. 
In Section \ref{sec:properties}, 
we review the properties of the theta functions. 
In Section \ref{sec:quintic}, 
we derive quintic versions of  the identities (\ref{Farkas-Kra-cubic-(1)}) and  (\ref{Farkas-Kra-cubic-(2)}). 
In particular, we believe that our results can be easily generalized to $k$ th power versions, $(k=7,9,11,\ldots).$
\par
In Section \ref{sec:rational(1)}, 
we prove that 
$\scriptsize{
\theta
\left[
\begin{array}{c}
1 \\
1/5
\end{array}
\right]
/
\theta
\left[
\begin{array}{c}
1 \\
3/5
\end{array}
\right]
}
$ 
is given by a rational expression of the following theta constants:
\begin{equation}
\label{eqn:theta-level5}
\theta
\left[
\begin{array}{c}
\frac{j}{5} \\
\frac{k}{5}
\end{array}
\right],
(j=1,3, k=1,3,5,7,9). 
\end{equation}
In Section \ref{sec:rational(2)}, 
we show that 
$
\scriptsize{
\theta
\left[
\begin{array}{c}
1 \\
1/5
\end{array}
\right],
}
$ or 
%%%%%%%%%%
$
\scriptsize{
\theta
\left[
\begin{array}{c}
1 \\
3/5
\end{array}
\right]
}
$ is given by rational expessions of the theta constants (\ref{eqn:theta-level5}). 
\par
In Section \ref{sec:rational(3)}, 
for $j=1,3,$ 
we prove that 
one of 
$\scriptsize{
\theta
\left[
\begin{array}{c}
j/5 \\
k/5
\end{array}
\right]
}, \,(k=1,3,5,7,9),
$ is given by two rational expressions of four of the following theta constants:
\begin{equation*}
\theta
\left[
\begin{array}{c}
\frac{j}{5} \\
\frac{k}{5}
\end{array}
\right],
(k=1,3,5,7,9). 
\end{equation*}
For this purpose, 
considering the determinant structure, 
we obtain the theta constant identities. 
This method is based on the method of Matsuda \cite{Matsuda}. 
\par
In Section \ref{sec:certain-theta}, 
for $j=1,3,$ 
we derive the theta constant identities among 
$\scriptsize{
\theta
\left[
\begin{array}{c}
j/5 \\
k/5
\end{array}
\right]
}, \,(k=1,3,5,7,9). 
$ 
In Section \ref{sec:algebraic-structure}, 
we investigate the algebraic structure of the theta constants (\ref{eqn:theta-level5}).

\subsubsection*{Remark}
By the identity (\ref{Farkas-Kra-cubic-(1)}), 
Farkas \cite{Farkas} showed that for each nonnegative integer $n\in\mathbb{N}_0,$ 
$$
\sigma(3n+2)
=
3\sum_{k=0}^n 
\delta(3k+1) \delta(3(n-k)+1),
$$
where for each positive integer $n\in\mathbb{N},$ 
$$
\sigma(n)=\sum_{d|n} d, \,\,\mathrm{and} \,\, \delta(n)=d_{1,3}(n)-d_{2,3}(n). 
$$
We believe that 
the theta constant identities of this paper can be applied to number theory. 
However, the computation is more complicated. 

\subsubsection*{Acknowledgments}
We are grateful to Professor H. Watanabe for his useful comments.

\section{The properties of the theta functions}
\label{sec:properties}
We first note that 
for $m,n\in\mathbb{Z},$ 
\begin{equation}
\label{eqn:integer-char}
\theta 
\left[
\begin{array}{c}
\epsilon \\
\epsilon^{\prime}
\end{array}
\right] (\zeta+n+m\tau, \tau) =
\exp(2\pi i)\left[\frac{n\epsilon-m\epsilon^{\prime}}{2}-mz-\frac{m^2\tau}{2}\right]
\theta 
\left[
\begin{array}{c}
\epsilon \\
\epsilon^{\prime}
\end{array}
\right] (\zeta,\tau),
\end{equation}
and 
\begin{equation}
\theta 
\left[
\begin{array}{c}
\epsilon +2m\\
\epsilon^{\prime}+2n
\end{array}
\right] 
(\zeta,\tau)
=\exp(\pi i \epsilon n)
\theta 
\left[
\begin{array}{c}
\epsilon \\
\epsilon^{\prime}
\end{array}
\right] 
(\zeta,\tau).
\end{equation}
Furthermore, 
\begin{equation*}
\theta 
\left[
\begin{array}{c}
-\epsilon \\
-\epsilon^{\prime}
\end{array}
\right] (\zeta,\tau)
=
\theta 
\left[
\begin{array}{c}
\epsilon \\
\epsilon^{\prime}
\end{array}
\right] (-\zeta,\tau)
\,\,
\mathrm{and}
\,\,
\theta^{\prime} 
\left[
\begin{array}{c}
-\epsilon \\
-\epsilon^{\prime}
\end{array}
\right] (\zeta,\tau)
=
-
\theta^{\prime} 
\left[
\begin{array}{c}
\epsilon \\
\epsilon^{\prime}
\end{array}
\right] (-\zeta,\tau).
\end{equation*}
%%%%%%%%%%%%%%%%
\par
For $m,n\in\mathbb{R},$ 
we see that 
\begin{align}
\label{eqn:real-char}
&\theta 
\left[
\begin{array}{c}
\epsilon \\
\epsilon^{\prime}
\end{array}
\right] \left(\zeta+\frac{n+m\tau}{2}, \tau\right)   \notag\\
&=
\exp(2\pi i)\left[
-\frac{m\zeta}{2}-\frac{m^2\tau}{8}-\frac{m(\epsilon^{\prime}+n)}{4}
\right]
\theta 
\left[
\begin{array}{c}
\epsilon+m \\
\epsilon^{\prime}+n
\end{array}
\right] 
(\zeta,\tau). 
\end{align}
%%%%%%%%%%%%%%%%%%%%%%%%%%%%
\par
We note that 
$\theta 
\left[
\begin{array}{c}
\epsilon \\
\epsilon^{\prime}
\end{array}
\right] \left(\zeta, \tau\right)$ has only one zero in the fundamental parallelogram, 
which is given by 
$$
\zeta=\frac{1-\epsilon}{2}\tau+\frac{1-\epsilon^{\prime}}{2}. 
$$
\par
All the theta functions have infinite product expansions, which are given by 
\begin{align}
\theta 
\left[
\begin{array}{c}
\epsilon \\
\epsilon^{\prime}
\end{array}
\right] (\zeta, \tau) &=\exp\left(\frac{\pi i \epsilon \epsilon^{\prime}}{2}\right) x^{\frac{\epsilon^2}{4}} z^{\frac{\epsilon}{2}} \notag \\
                           &\quad 
                           \displaystyle \prod_{n=1}^{\infty}(1-x^{2n})(1+e^{\pi i \epsilon^{\prime}} x^{2n-1+\epsilon} z)(1+e^{-\pi i \epsilon^{\prime}} x^{2n-1-\epsilon}/z),  \label{eqn:Jacobi}
\end{align}
where $x=\exp(\pi i \tau), \,z=\exp(2\pi i \zeta).$ 
\par
Following Farkas and Kra \cite{Farkas-Kra-1}, 
we last define 
$\mathcal{F}_{N}\left[
\begin{array}{c}
\epsilon \\
\epsilon^{\prime}
\end{array}
\right] $ to be the set of the entire functions $f$ that satisfy the two functional equations, 
$$
f(\zeta+1)=\exp(\pi i \epsilon) \,\,f(\zeta),
$$
and 
$$
f(\zeta+\tau)=\exp(-\pi i)[\epsilon^{\prime}+2N\zeta+N\tau] \,\,f(\zeta), \quad \zeta\in\mathbb{C},  \,\,\tau \in\mathbb{H}^2,
$$ 
where 
$N$ is a positive integer and 
$\left[
\begin{array}{c}
\epsilon \\
\epsilon^{\prime}
\end{array}
\right] \in\mathbb{R}^2.$ 
This set of functions is called the space of {\it $N$-th order $\theta$-functions with characteristic}
$\left[
\begin{array}{c}
\epsilon \\
\epsilon^{\prime}
\end{array}
\right]. $ 
Note that 
$$
\dim \mathcal{F}_{N}\left[
\begin{array}{c}
\epsilon \\
\epsilon^{\prime}
\end{array}
\right] =N.
$$
For its proof, see Farkas and Kra \cite[pp.131]{Farkas-Kra-1}.

\section{Quintic identities}
\label{sec:quintic}

\begin{theorem}
\label{thm:quintic-1/5,3/5-1/5,3/5,1,7/5,9/5}
{\it
For every $\tau\in\mathbb{H}^2, $ we have 
\begin{equation}
\label{eqn:quintic-1/5-1/5,3/5,1,7/5,9/5}
\theta^5
\left[
\begin{array}{c}
\frac15   \\
\frac15
\end{array}
\right]
-
\theta^5
\left[
\begin{array}{c}
\frac15   \\
\frac35
\end{array}
\right]
+
\theta^5
\left[
\begin{array}{c}
\frac15   \\
1
\end{array}
\right]
-
\theta^5
\left[
\begin{array}{c}
\frac15   \\
\frac75
\end{array}
\right]
+
\theta^5
\left[
\begin{array}{c}
\frac15   \\
\frac95
\end{array}
\right]
=0,
\end{equation}
and
\begin{equation}
\label{eqn:quintic-3/5-1/5,3/5,1,7/5,9/5}
\theta^5
\left[
\begin{array}{c}
\frac35   \\
\frac15
\end{array}
\right]
-
\theta^5
\left[
\begin{array}{c}
\frac35   \\
\frac35
\end{array}
\right]
+
\theta^5
\left[
\begin{array}{c}
\frac35   \\
1
\end{array}
\right]
-
\theta^5
\left[
\begin{array}{c}
\frac35   \\
\frac75
\end{array}
\right]
+
\theta^5
\left[
\begin{array}{c}
\frac35   \\
\frac95
\end{array}
\right]
=0.
\end{equation}
}
\end{theorem}

\begin{proof}
Consider the following elliptic functions:
\begin{equation*}
\varphi(z)
=
\frac
{
\theta^5
\left[
\begin{array}{c}
1 \\
1
\end{array}
\right](z)
}
{
%%%%
\theta
\left[
\begin{array}{c}
\frac15 \\
\frac15
\end{array}
\right](z)
%%%%%%%%%%%%
\theta
\left[
\begin{array}{c}
\frac15 \\
\frac35
\end{array}
\right](z)
%%%%%%%%%%%%%%%
\theta
\left[
\begin{array}{c}
\frac15 \\
1
\end{array}
\right](z)
%%%%%%%%%%%%%%%
\theta
\left[
\begin{array}{c}
\frac15 \\
\frac75
\end{array}
\right](z)
%%%%%%%%%%%%%%%
\theta
\left[
\begin{array}{c}
\frac15 \\
\frac95
\end{array}
\right](z)
}, 
\end{equation*}
and
\begin{equation*}
\psi(z)
=
\frac
{
\theta^5
\left[
\begin{array}{c}
1 \\
1
\end{array}
\right](z)
}
{
%%%%
\theta
\left[
\begin{array}{c}
\frac35 \\
\frac15
\end{array}
\right](z)
%%%%%%%%%%%%
\theta
\left[
\begin{array}{c}
\frac35 \\
\frac35
\end{array}
\right](z)
%%%%%%%%%%%%%%%
\theta
\left[
\begin{array}{c}
\frac35 \\
1
\end{array}
\right](z)
%%%%%%%%%%%%
\theta
\left[
\begin{array}{c}
\frac35 \\
\frac75
\end{array}
\right](z)
%%%%%%%%%%%%
\theta
\left[
\begin{array}{c}
\frac35 \\
\frac95
\end{array}
\right](z)
}. 
\end{equation*}
By $\varphi(z),$ 
we obtain equation (\ref{eqn:quintic-1/5-1/5,3/5,1,7/5,9/5}). 
Equation (\ref{eqn:quintic-3/5-1/5,3/5,1,7/5,9/5}) can be derived by $\psi(z)$ in the same way. 
\par
We first note that 
in the fundamental parallelogram, 
the poles of $\varphi(z)$ are $z=(2\tau+2)/5, (2\tau+1)/5, 2\tau/5, (2\tau-1)/5,$ and $(2\tau-2)/5.$ 
The direct calculation yields 
\begin{equation*}
\mathrm{Res}\left(\varphi(z), \frac{2\tau+2}{5} \right)
=
-
\frac
{
\theta^5
\left[
\begin{array}{c}
\frac15 \\
\frac15
\end{array}
\right]
%%%%%%%%%%%%%%%%
}
{
\theta^{\prime}
\left[
\begin{array}{c}
1 \\
1
\end{array}
\right]
%%%%%%%%%%%%%
\theta^2
\left[
\begin{array}{c}
1 \\
\frac15
\end{array}
\right]
%%%%%%%%%%%%%%%
\theta^2
\left[
\begin{array}{c}
1 \\
\frac35
\end{array}
\right]
},
\end{equation*}
%%%%%%%%%%%%%%%%%%%%%%%%%%%%%%%%%%
\begin{equation*}
\mathrm{Res}\left(\varphi(z), \frac{2\tau+1}{5} \right)
=
\frac
{
\theta^5
\left[
\begin{array}{c}
\frac15 \\
\frac35
\end{array}
\right]
%%%%%%%%%%%%%
}
{
\theta^{\prime}
\left[
\begin{array}{c}
1 \\
1
\end{array}
\right]
%%%%%%%%%%%%%
\theta^2
\left[
\begin{array}{c}
1 \\
\frac15
\end{array}
\right]
%%%%%%%%%%%%%
\theta^2
\left[
\begin{array}{c}
1 \\
\frac35
\end{array}
\right]
},
\end{equation*}
%%%%%%%%%%%%%%%%%%%%%%%%%%%%%%%%%%
%%%%%%%%%%%%%%%%%%%%%%%%%%%%%%%%%%%%
%and
%%%%%%%%%%%%%%%%%%%%%%%%%%%%%%%
%%%%%%%%%%%%%%%%%%%%%%%%%%%%%%%%%%%
\begin{equation*}
\mathrm{Res}\left(\varphi(z), \frac{2\tau}{5} \right)
=
-
\frac
{
\theta^5
\left[
\begin{array}{c}
\frac15 \\
1
\end{array}
\right]
%%%%%%%%%%%%%%%%
}
{
\theta^{\prime}
\left[
\begin{array}{c}
1 \\
1
\end{array}
\right]
\theta^2
\left[
\begin{array}{c}
1 \\
\frac15
\end{array}
\right]
\theta^2
\left[
\begin{array}{c}
1 \\
\frac35
\end{array}
\right]
},
\end{equation*}
%%%%%%%%%%%%%%%%%%%%%%%%%%%
\begin{equation*}
\mathrm{Res}\left(\varphi(z), \frac{2\tau-1}{5} \right)
=
\frac
{
\theta^5
\left[
\begin{array}{c}
\frac15 \\
\frac75
\end{array}
\right]
%%%%%%%%%%%%%
}
{
\theta^{\prime}
\left[
\begin{array}{c}
1 \\
1
\end{array}
\right]
%%%%%%%%%%%%%
\theta^2
\left[
\begin{array}{c}
1 \\
\frac15
\end{array}
\right]
%%%%%%%%%%%%%
\theta^2
\left[
\begin{array}{c}
1 \\
\frac35
\end{array}
\right]
},
\end{equation*}
and
\begin{equation*}
\mathrm{Res}\left(\varphi(z), \frac{2\tau-2}{5} \right)
=
-
\frac
{
\theta^5
\left[
\begin{array}{c}
\frac15 \\
\frac95
\end{array}
\right]
%%%%%%%%%%%%%
}
{
\theta^{\prime}
\left[
\begin{array}{c}
1 \\
1
\end{array}
\right]
%%%%%%%%%%%%%
\theta^2
\left[
\begin{array}{c}
1 \\
\frac15
\end{array}
\right]
%%%%%%%%%%%%%
\theta^2
\left[
\begin{array}{c}
1 \\
\frac35
\end{array}
\right]
}.
\end{equation*}
Since 
\begin{align*}
&
\mathrm{Res}\left(\varphi(z),(2\tau+2)/5 \right)
+
\mathrm{Res}\left(\varphi(z), (2\tau+1)/5 \right)
+
\mathrm{Res}\left(\varphi(z),2\tau/5 \right) \\
&\hspace{20mm}+
\mathrm{Res}\left(\varphi(z), (2\tau-1)/5 \right)
+
\mathrm{Res}\left(\varphi(z), (2\tau-2)/5 \right)
=0,
\end{align*}
equation (\ref{eqn:quintic-1/5-1/5,3/5,1,7/5,9/5}) follows. 
\end{proof}

\begin{theorem}
\label{thm:quintic-1/5,3/5,1,7/5,9/5-1/5,3/5}
{\it
For every $\tau\in\mathbb{H}^2, $ we have 
\begin{equation}
\label{thm:quintic-1/5,3/5,1,7/5,9/5-1/5}
\zeta_5
\theta^5
\left[
\begin{array}{c}
\frac15   \\
\frac15
\end{array}
\right]
+\zeta_5^3
\theta^5
\left[
\begin{array}{c}
\frac35   \\
\frac15
\end{array}
\right]
+
\theta^5
\left[
\begin{array}{c}
1   \\
\frac15
\end{array}
\right]
-
\zeta_5^2
\theta^5
\left[
\begin{array}{c}
\frac35   \\
\frac95
\end{array}
\right]
-
\zeta_5^4
\theta^5
\left[
\begin{array}{c}
\frac15   \\
\frac95
\end{array}
\right]
=0,
\end{equation}
and
\begin{equation}
\label{thm:quintic-1/5,3/5,1,7/5,9/5-3/5}
\zeta_5^3
\theta^5
\left[
\begin{array}{c}
\frac15   \\
\frac35
\end{array}
\right]
+
\zeta_5^4
\theta^5
\left[
\begin{array}{c}
\frac35   \\
\frac35
\end{array}
\right]
+
\theta^5
\left[
\begin{array}{c}
1   \\
\frac35
\end{array}
\right]
-
\zeta_5
\theta^5
\left[
\begin{array}{c}
\frac35   \\
\frac75
\end{array}
\right]
-
\zeta_5^2
\theta^5
\left[
\begin{array}{c}
\frac35   \\
\frac75
\end{array}
\right]
=0.
\end{equation}
}
\end{theorem}

\begin{proof}
Consider the following elliptic functions:
\begin{equation*}
\varphi(z)
=
\frac
{
\theta^5
\left[
\begin{array}{c}
1 \\
1
\end{array}
\right](z)
}
{
%%%%
\theta
\left[
\begin{array}{c}
\frac15 \\
\frac15
\end{array}
\right](z)
%%%%%%%%%%%%
\theta
\left[
\begin{array}{c}
\frac35 \\
\frac15
\end{array}
\right](z)
%%%%%%%%%%%%%%%
\theta
\left[
\begin{array}{c}
1 \\
\frac15
\end{array}
\right](z)
%%%%%%%%%%%%%%%
\theta
\left[
\begin{array}{c}
\frac75 \\
\frac15
\end{array}
\right](z)
%%%%%%%%%%%%%%%
\theta
\left[
\begin{array}{c}
\frac95 \\
\frac15
\end{array}
\right](z)
}, 
\end{equation*}
and
\begin{equation*}
\psi(z)
=
\frac
{
\theta^5
\left[
\begin{array}{c}
1 \\
1
\end{array}
\right](z)
}
{
%%%%
\theta
\left[
\begin{array}{c}
\frac15 \\
\frac35
\end{array}
\right](z)
%%%%%%%%%%%%
\theta
\left[
\begin{array}{c}
\frac35 \\
\frac35
\end{array}
\right](z)
%%%%%%%%%%%%%%%
\theta
\left[
\begin{array}{c}
1 \\
\frac35
\end{array}
\right](z)
%%%%%%%%%%%%
\theta
\left[
\begin{array}{c}
\frac75 \\
\frac35
\end{array}
\right](z)
%%%%%%%%%%%%
\theta
\left[
\begin{array}{c}
\frac95 \\
\frac35
\end{array}
\right](z)
}. 
\end{equation*}
The theorem can be proved in the same way as Theorem \ref{thm:quintic-1/5,3/5-1/5,3/5,1,7/5,9/5}.  
\end{proof}

\subsection*{Rermark}
The identity (\ref{eqn:quintic-1/5-1/5,3/5,1,7/5,9/5}) was proved by Farkas and Kra \cite[pp. 277]{Farkas-Kra-1}. 
We believe that the other identities in Theorems \ref{thm:quintic-1/5,3/5-1/5,3/5,1,7/5,9/5} and 
\ref{thm:quintic-1/5,3/5,1,7/5,9/5-1/5,3/5} are new.

\section{Rational expressions of theta constants (1)}  
\label{sec:rational(1)}

%%%%%%%%%%%%%%%%%%%%%%%%%%%%%%%%%%%%%%%%%%%%%%%%%%%%%%%
%%%%%%%%%%%%%%%%%%%%%%%%%%%%%%%%%%%%%%%%%%%%%%%%%%%%%
%%%%%%%%%%%%%%%%%%%%%%%%%%%%%%%%%%%%%%%%%%%%%%%%%%%%%%%%%
\subsection{The case for $j=1$ }

In this subsection, we prove the following theorem:

\begin{theorem}
\label{thm:main-rational-(1)}
{\it
$
\scriptsize
{
\theta
\left[
\begin{array}{c}
1 \\
1/5
\end{array}
\right]
/
\theta
\left[
\begin{array}{c}
1 \\
3/5
\end{array}
\right]
}
$ is given by a rational expression of four of 
$
\scriptsize
{
\theta
\left[
\begin{array}{c}
1/5 \\
k/5
\end{array}
\right], 
}
$
where 
$k=1,3,5,7,9.$
}
\end{theorem}

For the proof, we obtain the following proposition and theorem:

\begin{proposition}
\label{prop:three-theta-1/5}
{\it
\normalsize{
For every $(\zeta,\tau)\in\mathbb{C}\times\mathbb{H}^2,$
we have 
}
\small{
\begin{align}
&\theta
\left[
\begin{array}{c}
1 \\
\frac35
\end{array}
\right]
%%%%%%
\theta^2 
\left[
\begin{array}{c}
\frac15 \\
\frac15
\end{array}
\right](\zeta,\tau)
%%%%%%%%%%%%%
\theta
\left[
\begin{array}{c}
\frac15 \\
\frac35
\end{array}
\right](\zeta,\tau)
%%%%%%
%%%%%%%%
+\zeta_5^2
\theta
\left[
\begin{array}{c}
1 \\
\frac15
\end{array}
\right]
%%%%%%
\theta^2 
\left[
\begin{array}{c}
\frac15 \\
\frac35
\end{array}
\right](\zeta,\tau)
%%%%%%%%%%%%%
\theta
\left[
\begin{array}{c}
\frac15 \\
\frac95
\end{array}
\right](\zeta,\tau)   \notag \\
%%%%%%%%%%%%%%%
%%%%%%%%%%%
-&
\zeta_5^4
\theta
\left[
\begin{array}{c}
1 \\
\frac35
\end{array}
\right]
%%%%%%
\theta^2 
\left[
\begin{array}{c}
\frac15 \\
\frac95
\end{array}
\right](\zeta,\tau)
%%%%%%%%%%%%%
\theta
\left[
\begin{array}{c}
\frac15 \\
\frac75
\end{array}
\right](\zeta,\tau) 
%%%%%%%%%%%%%%%%%%%%%
%%%%%%%%%%%%%%%%%%
-\zeta_5^2
\theta
\left[
\begin{array}{c}
1 \\
\frac15
\end{array}
\right]
%%%%%%
\theta^2 
\left[
\begin{array}{c}
\frac15 \\
\frac75
\end{array}
\right](\zeta,\tau)
%%%%%%%%%%%%%
\theta
\left[
\begin{array}{c}
\frac15 \\
\frac15
\end{array}
\right](\zeta,\tau) =0,  \label{eqn:three-theta-1/5-5/5}
\end{align}
%%%%%%%%%%%%%%%%%%%%%%%%%%%%%%%%%%%%%%%%%%%%%%%
%%%%%%%%%%%%%%%%%%%%%%%%%%%%%%%%%%%%%%%%%%%%%%%%%%%
%%%%%%%%%%%%%%%%%%%%%%%%%%%%%%%%%%%%%%%%%%%%%%%%%%%%%%%
%%%%%%%%%%%%%%%%%%%%%%%%%%%%%%%%%%%%%%%%%%%%%%%%%%%%%%%%%
\begin{align}
&\theta
\left[
\begin{array}{c}
1 \\
\frac15
\end{array}
\right]
%%%%%%
\theta^2 
\left[
\begin{array}{c}
\frac15 \\
\frac15
\end{array}
\right](\zeta,\tau)
%%%%%%%%%%%%%
\theta
\left[
\begin{array}{c}
\frac15 \\
1
\end{array}
\right](\zeta,\tau)
%%%%%%
%%%%%%%%
-\zeta_5^2
\theta
\left[
\begin{array}{c}
1 \\
\frac35
\end{array}
\right]
%%%%%%
\theta^2 
\left[
\begin{array}{c}
\frac15 \\
1
\end{array}
\right](\zeta,\tau)
%%%%%%%%%%%%%
\theta
\left[
\begin{array}{c}
\frac15 \\
\frac75
\end{array}
\right](\zeta,\tau)   \notag \\
%%%%%%%%%%%%%%%
%%%%%%%%%%%
+&
\zeta_5^2
\theta
\left[
\begin{array}{c}
1 \\
\frac15
\end{array}
\right]
%%%%%%
\theta^2 
\left[
\begin{array}{c}
\frac15 \\
\frac75
\end{array}
\right](\zeta,\tau)
%%%%%%%%%%%%%
\theta
\left[
\begin{array}{c}
\frac15 \\
\frac35
\end{array}
\right](\zeta,\tau) 
%%%%%%%%%%%%%%%%%%%%%
%%%%%%%%%%%%%%%%%%
-
\theta
\left[
\begin{array}{c}
1 \\
\frac35
\end{array}
\right]
%%%%%%
\theta^2 
\left[
\begin{array}{c}
\frac15 \\
\frac35
\end{array}
\right](\zeta,\tau)
%%%%%%%%%%%%%
\theta
\left[
\begin{array}{c}
\frac15 \\
\frac15
\end{array}
\right](\zeta,\tau) =0,  \label{eqn:three-theta-1/5-7/5}
\end{align}
%%%%%%%%%%%%%%%%%%%%%%%%%%%%%%%%%%%%%%%%%%%%%%%
%%%%%%%%%%%%%%%%%%%%%%%%%%%%%%%%%%%%%%%%%%%%%%%%%%%
%%%%%%%%%%%%%%%%%%%%%%%%%%%%%%%%%%%%%%%%%%%%%%%%%%%%%%%
%%%%%%%%%%%%%%%%%%%%%%%%%%%%%%%%%%%%%%%%%%%%%%%%%%%%%%%%%
\begin{align}
&\theta
\left[
\begin{array}{c}
1 \\
\frac15
\end{array}
\right]
%%%%%%
\theta^2 
\left[
\begin{array}{c}
\frac15 \\
\frac15
\end{array}
\right](\zeta,\tau)
%%%%%%%%%%%%%
\theta
\left[
\begin{array}{c}
\frac15 \\
\frac75
\end{array}
\right](\zeta,\tau)
%%%%%%
%%%%%%%%
-\zeta_5^2
\theta
\left[
\begin{array}{c}
1 \\
\frac35
\end{array}
\right]
%%%%%%
\theta^2 
\left[
\begin{array}{c}
\frac15 \\
\frac75
\end{array}
\right](\zeta,\tau)
%%%%%%%%%%%%%
\theta
\left[
\begin{array}{c}
\frac15 \\
1
\end{array}
\right](\zeta,\tau)    \notag \\
%%%%%%%%%%%%%%%
%%%%%%%%%%%
+&
\zeta_5^2
\theta
\left[
\begin{array}{c}
1 \\
\frac15
\end{array}
\right]
%%%%%%
\theta^2 
\left[
\begin{array}{c}
\frac15 \\
1
\end{array}
\right](\zeta,\tau)
%%%%%%%%%%%%%
\theta
\left[
\begin{array}{c}
\frac15 \\
\frac95
\end{array}
\right](\zeta,\tau) 
%%%%%%%%%%%%%%%%%%%%%
%%%%%%%%%%%%%%%%%%
+\zeta_5^2
\theta
\left[
\begin{array}{c}
1 \\
\frac35
\end{array}
\right]
%%%%%%
\theta^2 
\left[
\begin{array}{c}
\frac15 \\
\frac95
\end{array}
\right](\zeta,\tau)
%%%%%%%%%%%%%
\theta
\left[
\begin{array}{c}
\frac15 \\
\frac15
\end{array}
\right](\zeta,\tau) =0,  \label{eqn:three-theta-1/5-9/5}
\end{align}
%%%%%%%%%%%%%%%%%%%%%%%%%%%%%%%%%%%%%%%%%%%%%%%
%%%%%%%%%%%%%%%%%%%%%%%%%%%%%%%%%%%%%%%%%%%%%%%%%%%
%%%%%%%%%%%%%%%%%%%%%%%%%%%%%%%%%%%%%%%%%%%%%%%%%%%%%%%
%%%%%%%%%%%%%%%%%%%%%%%%%%%%%%%%%%%%%%%%%%%%%%%%%%%%%%%%%
\begin{align}
&\theta
\left[
\begin{array}{c}
1 \\
\frac35
\end{array}
\right]
%%%%%%
\theta^2 
\left[
\begin{array}{c}
\frac15 \\
\frac15
\end{array}
\right](\zeta,\tau)
%%%%%%%%%%%%%
\theta
\left[
\begin{array}{c}
\frac15 \\
\frac95
\end{array}
\right](\zeta,\tau)
%%%%%%
%%%%%%%%
+\zeta_5^2
\theta
\left[
\begin{array}{c}
1 \\
\frac15
\end{array}
\right]
%%%%%%
\theta^2 
\left[
\begin{array}{c}
\frac15 \\
\frac95
\end{array}
\right](\zeta,\tau)
%%%%%%%%%%%%%
\theta
\left[
\begin{array}{c}
\frac15 \\
\frac35
\end{array}
\right](\zeta,\tau)    \notag \\
%%%%%%%%%%%%%%%
%%%%%%%%%%%
-&
\theta
\left[
\begin{array}{c}
1 \\
\frac35
\end{array}
\right]
%%%%%%
\theta^2 
\left[
\begin{array}{c}
\frac15 \\
\frac35
\end{array}
\right](\zeta,\tau)
%%%%%%%%%%%%%
\theta
\left[
\begin{array}{c}
\frac15 \\
1
\end{array}
\right](\zeta,\tau) 
%%%%%%%%%%%%%%%%%%%%%
%%%%%%%%%%%%%%%%%%
+
\theta
\left[
\begin{array}{c}
1 \\
\frac15
\end{array}
\right]
%%%%%%
\theta^2 
\left[
\begin{array}{c}
\frac15 \\
1
\end{array}
\right](\zeta,\tau)
%%%%%%%%%%%%%
\theta
\left[
\begin{array}{c}
\frac15 \\
\frac15
\end{array}
\right](\zeta,\tau) =0,  \label{eqn:three-theta-1/5-1/5}
\end{align}
%%%%%%%%%%%%%%%%%%%%%%%%%%%%%%%%%%%%%%%%%%%%%%%
%%%%%%%%%%%%%%%%%%%%%%%%%%%%%%%%%%%%%%%%%%%%%%%%%%%
%%%%%%%%%%%%%%%%%%%%%%%%%%%%%%%%%%%%%%%%%%%%%%%%%%%%%%%
%%%%%%%%%%%%%%%%%%%%%%%%%%%%%%%%%%%%%%%%%%%%%%%%%%%%%%%%%
\begin{align}
&\theta
\left[
\begin{array}{c}
1 \\
\frac15
\end{array}
\right]
%%%%%%
\theta^2 
\left[
\begin{array}{c}
\frac15 \\
\frac35
\end{array}
\right](\zeta,\tau)
%%%%%%%%%%%%%
\theta
\left[
\begin{array}{c}
\frac15 \\
\frac75
\end{array}
\right](\zeta,\tau)
%%%%%%
%%%%%%%%
-\zeta_5^2
\theta
\left[
\begin{array}{c}
1 \\
\frac35
\end{array}
\right]
%%%%%%
\theta^2 
\left[
\begin{array}{c}
\frac15 \\
\frac75
\end{array}
\right](\zeta,\tau)
%%%%%%%%%%%%%
\theta
\left[
\begin{array}{c}
\frac15 \\
\frac95
\end{array}
\right] (\zeta,\tau)   \notag \\
%%%%%%%%%%%%%%%
%%%%%%%%%%%
+&\zeta_5^2
\theta
\left[
\begin{array}{c}
1 \\
\frac15
\end{array}
\right]
%%%%%%
\theta^2 
\left[
\begin{array}{c}
\frac15 \\
\frac95
\end{array}
\right](\zeta,\tau)
%%%%%%%%%%%%%
\theta
\left[
\begin{array}{c}
\frac15 \\
1
\end{array}
\right](\zeta,\tau)
%%%%%%%%%%%%%%%%%%%%%
%%%%%%%%%%%%%%%%%%
-
\theta
\left[
\begin{array}{c}
1 \\
\frac35
\end{array}
\right]
%%%%%%
\theta^2 
\left[
\begin{array}{c}
\frac15 \\
1
\end{array}
\right](\zeta,\tau)
%%%%%%%%%%%%%
\theta
\left[
\begin{array}{c}
\frac15 \\
\frac35
\end{array}
\right](\zeta,\tau) =0. \label{eqn:three-theta-1/5-3/5}
\end{align}
}
}
\end{proposition}

\begin{proof}
We treat equation (\ref{eqn:three-theta-1/5-5/5}). 
The others can be proved in the same way. 
For this purpose, 
we first note that 
$
\dim \mathcal{F}_{3}\left[
\begin{array}{c}
\frac35 \\
1
\end{array}
\right] =3
$ 
and 
\begin{align*}
&
\theta^2 
\left[
\begin{array}{c}
\frac15 \\
\frac15
\end{array}
\right](\zeta,\tau)
%%%%%%%%%%%%%
\theta
\left[
\begin{array}{c}
\frac15 \\
\frac35
\end{array}
\right](\zeta,\tau), \,\,
%%%%%%%%%%%%%%%%%%%%%%%%%%%%%%%%%%
%%%%%%%%%%%%%%%%%%%%%%%%%%%%%%%%%%%
\theta^2 
\left[
\begin{array}{c}
\frac15 \\
\frac35
\end{array}
\right](\zeta,\tau)
%%%%%%%%%%%%%
\theta
\left[
\begin{array}{c}
\frac15 \\
\frac95
\end{array}
\right](\zeta,\tau)  \\
%%%%%%%%%%%%%%%%%%%%%%%%%%%%
%%%%%%%%%%%%%%%%%%%%%%%%%%%
&
\theta^2 
\left[
\begin{array}{c}
\frac15 \\
\frac95
\end{array}
\right](\zeta,\tau)
%%%%%%%%%%%%%
\theta
\left[
\begin{array}{c}
\frac15 \\
\frac75
\end{array}
\right](\zeta,\tau), \,\,
%%%%%%%%%%%%%%%%%%%%%%%%%%%%%%%
%%%%%%%%%%%%%%%%%%%%%%%%%%%%%%%
\theta^2 
\left[
\begin{array}{c}
\frac15 \\
\frac75
\end{array}
\right](\zeta,\tau)
%%%%%%%%%%%%%
\theta
\left[
\begin{array}{c}
\frac15 \\
\frac15
\end{array}
\right](\zeta,\tau) 
%%%%%%%%%%%%%%%%%%%%%%%%%%%
%%%%%%%%%%%%%%%%%%%%%%%%%%%%
\in
%%%%%%%%%%%%%%%%%
\mathcal{F}_{3}\left[
\begin{array}{c}
\frac35 \\
1
\end{array}
\right]. 
\end{align*}
Therefore, 
there exist complex numbers, $x_1, x_2, x_3, x_4,$ not all zero 
such that 
\begin{align*}
&
x_1
%%%%%%%%%%%%%%%%
\theta^2 
\left[
\begin{array}{c}
\frac15 \\
\frac15
\end{array}
\right](\zeta,\tau)
%%%%%%%%%%%%%
\theta
\left[
\begin{array}{c}
\frac15 \\
\frac35
\end{array}
\right](\zeta,\tau)
%%%%%%%%%%%%%%%%%%%%%%%%%%%%%%%%%%
%%%%%%%%%%%%%%%%%%%%%%%%%%%%%%%%%%%
+
%%%%%%%%%%%%%%%%%%%%%%%%%%%%%%%%%%
%%%%%%%%%%%%%%%%%%%%%%%%%%%%%%%%%%
x_2
%%%%%%%%%%%%%
\theta^2 
\left[
\begin{array}{c}
\frac15 \\
\frac35
\end{array}
\right](\zeta,\tau)
%%%%%%%%%%%%%
\theta
\left[
\begin{array}{c}
\frac15 \\
\frac95
\end{array}
\right](\zeta,\tau)  \\
%%%%%%%%%%%%%%%%%%%%%%%%%%%%
%%%%%%%%%%%%%%%%%%%%%%%%%%%
&+
%%%%%%%%%%%%%%%%%%%%%%%%%%%%%%
%%%%%%%%%%%%%%%%%%%%%%%%%%%%%%
x_3
%%%%%%%%%%%%%%%
\theta^2 
\left[
\begin{array}{c}
\frac15 \\
\frac95
\end{array}
\right](\zeta,\tau)
%%%%%%%%%%%%%
\theta
\left[
\begin{array}{c}
\frac15 \\
\frac75
\end{array}
\right](\zeta,\tau)
%%%%%%%%%%%%%%%%%%%%%%%%%%%%%%%
%%%%%%%%%%%%%%%%%%%%%%%%%%%%%%%
+
%%%%%%%%%%%%%%%%%%%%%%%%%%%%%%%%
%%%%%%%%%%%%%%%%%%%%%%%%%%%%%%%%%%
x_4
%%%%%%%%%%%%%%%%%%
\theta^2 
\left[
\begin{array}{c}
\frac15 \\
\frac75
\end{array}
\right](\zeta,\tau)
%%%%%%%%%%%%%
\theta
\left[
\begin{array}{c}
\frac15 \\
\frac15
\end{array}
\right](\zeta,\tau)=0. 
\end{align*}
By substituting 
$$
\zeta=
\frac{2\tau\pm2}{5}, \frac{2\tau\pm1}{5}, 
$$
we have the following 
system of equations:
\begin{alignat*}{5}
&  
&  
&x_2
\zeta_5^2
%%%%%%%%%%
\theta
\left[
\begin{array}{c}
1 \\
\frac35
\end{array}
\right]
&
+&x_3
%%%%%%%
\theta
\left[
\begin{array}{c}
1 \\
\frac15
\end{array}
\right]
&
&
&
&=0,  \\
%%%%%%%%%%%%%%%%%%%%%%%
%%%%%%%%%%%%%%%%%%%%%
&  
&  
&
&
&x_3
%%%%%%%
\theta
\left[
\begin{array}{c}
1 \\
\frac15
\end{array}
\right]
&
-&x_4
\zeta_5^2
%%%%%%%%%%
\theta
\left[
\begin{array}{c}
1 \\
\frac35
\end{array}
\right]
&
&=0,  \\
%%%%%%%%%%%%%%%%%%%%%%%
%%%%%%%%%%%%%%%%%%%%%
&x_1
%%%%%%%
\zeta_5^2
%%%%%%%%%%%%
\theta
\left[
\begin{array}{c}
1 \\
\frac15
\end{array}
\right]  
&  
&
&
&
&
+&x_4
%%%%%%%%%%
\theta
\left[
\begin{array}{c}
1 \\
\frac35
\end{array}
\right]
&
&=0,  \\
%%%%%%%%%%%%%%%%%%%%%%%
%%%%%%%%%%%%%%%%%%%%%
&x_1
%%%%%%%
\zeta_5^2
%%%%%%%%%%%%
\theta
\left[
\begin{array}{c}
1 \\
\frac15
\end{array}
\right]  
& 
-&x_2
%%%%%%%%%%
\theta
\left[
\begin{array}{c}
1 \\
\frac35
\end{array}
\right]
& 
&
&
&
&
&=0.  \\
\end{alignat*}
Solving this system of equations, 
we obtain
$$
(x_1, x_2, x_3, x_4)
=
\alpha
\left(
\theta
\left[
\begin{array}{c}
1 \\
\frac35
\end{array}
\right], 
%%%%%%%%%%%%%%%%
%%%%%%%%%%%%%%%%
\zeta_5^2
%%%%%%%%%%%%
\theta
\left[
\begin{array}{c}
1 \\
\frac15
\end{array}
\right], 
%%%%%%%%%%%%%%%%%%%
%%%%%%%%%%%%%%%%%%%%
-
\zeta_5^4
%%%%%%%%%%%%%%%
\theta
\left[
\begin{array}{c}
1 \\
\frac35
\end{array}
\right], 
%%%%%%%%%%%%%%%%
%%%%%%%%%%%%%%%%
-
\zeta_5^2
%%%%%%%%%%%%
\theta
\left[
\begin{array}{c}
1 \\
\frac15
\end{array}
\right]
\right),
$$
for some nonzero complex number $\alpha\in\mathbb{C}^{*},$ 
which proves the proposition.  
\end{proof}

\begin{theorem}
\label{thm:for-main(1)}
{\it
For every $\tau\in\mathbb{H}^2,$ 
we have
\begin{equation}
\label{eqn:rational-(1/5,1)delete}       
\frac
{
\theta
\left[
\begin{array}{c}
1 \\
\frac15
\end{array}
\right](0,\tau)
}
{
\theta
\left[
\begin{array}{c}
1 \\
\frac35
\end{array}
\right](0,\tau)
}
=
\frac
{
\theta^2
\left[
\begin{array}{c}
\frac15 \\
\frac15
\end{array}
\right](0,\tau)
%%%%%%%%%%%%%
\theta
\left[
\begin{array}{c}
\frac15 \\
\frac35
\end{array}
\right](0,\tau)
%%%%%%%%%%%%%%%
%%%%%%%%%%%%%%%%
-
%%%%%%%%%%%%%%%%%
%%%%%%%%%%%%%%%%%%%
\zeta_5^4
%%%%%%%%%%%%%
\theta^2
\left[
\begin{array}{c}
\frac15 \\
\frac95
\end{array}
\right](0,\tau)
%%%%%%%%%%%%
\theta
\left[
\begin{array}{c}
\frac15 \\
\frac75
\end{array}
\right](0,\tau)
}
{
\zeta_5^2
\theta^2
\left[
\begin{array}{c}
\frac15 \\
\frac75
\end{array}
\right](0,\tau)
%%%%%%%%%%%%%
\theta
\left[
\begin{array}{c}
\frac15 \\
\frac15
\end{array}
\right](0,\tau)
%%%%%%%%%%%%%%%
%%%%%%%%%%%%%%%%
-
%%%%%%%%%%%%%%%%%
%%%%%%%%%%%%%%%%%%%
\zeta_5^2
%%%%%%%%%%
\theta^2
\left[
\begin{array}{c}
\frac15 \\
\frac35
\end{array}
\right](0,\tau)
%%%%%%%%%%%%
\theta
\left[
\begin{array}{c}
\frac15 \\
\frac95
\end{array}
\right](0,\tau)
},         
\end{equation}
%%%%%%%%%%%%%%%%%%%%%%%%%%%
%%%%%%%%%%%%%%%%%%%%%%%%%%
%%%%%%%%%%%%%%%%%%%%%%%%%
%%%%%%%%%%%%%%%%%%%%%%%%%%%
%%%%%%%%%%%%%%%%%%%%%%%%%%
\begin{equation}
\label{eqn:rational-(1/5,9/5)delete}  
\frac
{
\theta
\left[
\begin{array}{c}
1 \\
\frac15
\end{array}
\right](0,\tau)
}
{
\theta
\left[
\begin{array}{c}
1 \\
\frac35
\end{array}
\right](0,\tau)
}
=
\frac
{
\zeta_5^2
%%%%%%%%%%
\theta^2
\left[
\begin{array}{c}
\frac15 \\
1
\end{array}
\right](0,\tau)
%%%%%%%%%%%%%
\theta
\left[
\begin{array}{c}
\frac15 \\
\frac75
\end{array}
\right](0,\tau)
%%%%%%%%%%%%%%%
%%%%%%%%%%%%%%%%
+
%%%%%%%%%%%%%%%%%
%%%%%%%%%%%%%%%%%%%
\theta^2
\left[
\begin{array}{c}
\frac15 \\
\frac35
\end{array}
\right](0,\tau)
%%%%%%%%%%%%
\theta
\left[
\begin{array}{c}
\frac15 \\
\frac15
\end{array}
\right](0,\tau)
}
{
\theta^2
\left[
\begin{array}{c}
\frac15 \\
\frac15
\end{array}
\right](0,\tau)
%%%%%%%%%%%%%
\theta
\left[
\begin{array}{c}
\frac15 \\
1
\end{array}
\right](0,\tau)
%%%%%%%%%%%%%%%
%%%%%%%%%%%%%%%%
+
%%%%%%%%%%%%%%%%%
%%%%%%%%%%%%%%%%%%%
\zeta_5^2
%%%%%%%%%%
\theta^2
\left[
\begin{array}{c}
\frac15 \\
\frac75
\end{array}
\right](0,\tau)
%%%%%%%%%%%%
\theta
\left[
\begin{array}{c}
\frac15 \\
\frac35
\end{array}
\right](0,\tau)
},
\end{equation}               
%%%%%%%%%%%%%%%%%%%%%%%%%%%
%%%%%%%%%%%%%%%%%%%%%%%%%%
%%%%%%%%%%%%%%%%%%%%%%%%%
%%%%%%%%%%%%%%%%%%%%%%%%%%%
%%%%%%%%%%%%%%%%%%%%%%%%%%
\begin{equation}
\label{eqn:rational-(1/5,3/5)delete}  
\frac
{
\theta
\left[
\begin{array}{c}
1 \\
\frac15
\end{array}
\right](0,\tau)
}
{
\theta
\left[
\begin{array}{c}
1 \\
\frac35
\end{array}
\right](0,\tau)
}
=
\frac
{
\zeta_5^2
%%%%%%%%%%
\theta^2
\left[
\begin{array}{c}
\frac15 \\
\frac75
\end{array}
\right](0,\tau)
%%%%%%%%%%%%%
\theta
\left[
\begin{array}{c}
\frac15 \\
1
\end{array}
\right](0,\tau)
%%%%%%%%%%%%%%%
%%%%%%%%%%%%%%%%
-
%%%%%%%%%%%%%%%%%
%%%%%%%%%%%%%%%%%%%
\zeta_5^2
%%%%%%%%%%
\theta^2
\left[
\begin{array}{c}
\frac15 \\
\frac95
\end{array}
\right](0,\tau)
%%%%%%%%%%%%
\theta
\left[
\begin{array}{c}
\frac15 \\
\frac15
\end{array}
\right](0,\tau)
}
{
\theta^2
\left[
\begin{array}{c}
\frac15 \\
\frac15
\end{array}
\right](0,\tau)
%%%%%%%%%%%%%
\theta
\left[
\begin{array}{c}
\frac15 \\
\frac75
\end{array}
\right](0,\tau)
%%%%%%%%%%%%%%%
%%%%%%%%%%%%%%%%
+
%%%%%%%%%%%%%%%%%
%%%%%%%%%%%%%%%%%%%
\zeta_5^2
%%%%%%%%%%
\theta^2
\left[
\begin{array}{c}
\frac15 \\
1
\end{array}
\right](0,\tau)
%%%%%%%%%%%%
\theta
\left[
\begin{array}{c}
\frac15 \\
\frac95
\end{array}
\right](0,\tau)
},
\end{equation}             
%%%%%%%%%%%%%%%%%%%%%%%%%%%%%%
%%%%%%%%%%%%%%%%%%%%%%%%%%
%%%%%%%%%%%%%%%%%%%%%%%%%
%%%%%%%%%%%%%%%%%%%%%%%%%%%
%%%%%%%%%%%%%%%%%%%%%%%%%%%%%%
\begin{equation}
\label{eqn:rational-(1/5,7/5)delete}    
\frac
{
\theta
\left[
\begin{array}{c}
1 \\
\frac15
\end{array}
\right](0,\tau)
}
{
\theta
\left[
\begin{array}{c}
1 \\
\frac35
\end{array}
\right](0,\tau)
}
=
\frac
{
-
%%%%%%%%%%
\theta^2
\left[
\begin{array}{c}
\frac15 \\
\frac15
\end{array}
\right](0,\tau)
%%%%%%%%%%%%%
\theta
\left[
\begin{array}{c}
\frac15 \\
\frac95
\end{array}
\right](0,\tau)
%%%%%%%%%%%%%%%
%%%%%%%%%%%%%%%%
+
%%%%%%%%%%%%%%%%%
%%%%%%%%%%%%%%%%%%%
\theta^2
\left[
\begin{array}{c}
\frac15 \\
\frac35
\end{array}
\right](0,\tau)
%%%%%%%%%%%%
\theta
\left[
\begin{array}{c}
\frac15 \\
1
\end{array}
\right](0,\tau)
}
{
\zeta_5^2
%%%%%%%%%%%
\theta^2
\left[
\begin{array}{c}
\frac15 \\
\frac95
\end{array}
\right](0,\tau)
%%%%%%%%%%%%%
\theta
\left[
\begin{array}{c}
\frac15 \\
\frac35
\end{array}
\right](0,\tau)
%%%%%%%%%%%%%%%
%%%%%%%%%%%%%%%%
+
%%%%%%%%%%%%%%%%%
%%%%%%%%%%%%%%%%%%%
\theta^2
\left[
\begin{array}{c}
\frac15 \\
1
\end{array}
\right](0,\tau)
%%%%%%%%%%%%
\theta
\left[
\begin{array}{c}
\frac15 \\
\frac15
\end{array}
\right](0,\tau)
},
\end{equation} 
%%%%%%%%%%%%%%%%%%%%%%%%%
%%%%%%%%%%%%%%%%%%%%%%%%%%
%%%%%%%%%%%%%%%%%%%%%%%%%
%%%%%%%%%%%%%%%%%%%%%%%%%%%
%%%%%%%%%%%%%%%%%%%%%%%%%%%
\begin{equation}
\label{eqn:rational-(1/5,1/5)delete}
\frac
{
\theta
\left[
\begin{array}{c}
1 \\
\frac15
\end{array}
\right](0,\tau)
}
{
\theta
\left[
\begin{array}{c}
1 \\
\frac35
\end{array}
\right](0,\tau)
}
=
\frac
{
\zeta_5^2
%%%%%%%%%%
\theta^2
\left[
\begin{array}{c}
\frac15 \\
\frac75
\end{array}
\right](0,\tau)
%%%%%%%%%%%%%
\theta
\left[
\begin{array}{c}
\frac15 \\
\frac95
\end{array}
\right](0,\tau)
%%%%%%%%%%%%%%%
%%%%%%%%%%%%%%%%
+
%%%%%%%%%%%%%%%%%
%%%%%%%%%%%%%%%%%%%
\theta^2
\left[
\begin{array}{c}
\frac15 \\
1
\end{array}
\right](0,\tau)
%%%%%%%%%%%%
\theta
\left[
\begin{array}{c}
\frac15 \\
\frac35
\end{array}
\right](0,\tau)
}
{
\theta^2
\left[
\begin{array}{c}
\frac15 \\
\frac35
\end{array}
\right](0,\tau)
%%%%%%%%%%%%%
\theta
\left[
\begin{array}{c}
\frac15 \\
\frac75
\end{array}
\right](0,\tau)
%%%%%%%%%%%%%%%
%%%%%%%%%%%%%%%%
+
%%%%%%%%%%%%%%%%%
%%%%%%%%%%%%%%%%%%%
\zeta_5^2
%%%%%%%%%%
\theta^2
\left[
\begin{array}{c}
\frac15 \\
\frac95
\end{array}
\right](0,\tau)
%%%%%%%%%%%%
\theta
\left[
\begin{array}{c}
\frac15 \\
1
\end{array}
\right](0,\tau)
}.             
\end{equation}
}
\end{theorem}
%%%%%%%%%%%%%%%%%%%%%%%%%%%%%%%%%%%%%%%%%%%%%%%%%%%%%%%
%%%%%%%%%%%%%%%%%%%%%%%%%%%%%%%%%%%%%%%%%%%%%%%%%%%%%
%%%%%%%%%%%%%%%%%%%%%%%%%%%%%%%%%%%%%%%%%%%%%%%%%%%%%%%%%

\begin{proof}
We deal with equation (\ref{eqn:rational-(1/5,1)delete}). 
The others can be proved in the same way. 
For this purpose, 
setting $\zeta=0$ in equation (\ref{eqn:three-theta-1/5-5/5}), 
we have
\begin{align*}
&
\theta
\left[
\begin{array}{c}
1 \\
\frac35
\end{array}
\right](0,\tau)
%%%%%%%%%%%%%%
\Bigg\{
\theta^2
\left[
\begin{array}{c}
\frac15 \\
\frac15
\end{array}
\right](0,\tau)
%%%%%%%%%%%%%%
\theta
\left[
\begin{array}{c}
\frac15 \\
\frac35
\end{array}
\right](0,\tau)
%%%%%%%%%%%%%%%%%%%%%%%%%%%%
%%%%%%%%%%%%%%%%%%%%%%%%%%%%%
-
%%%%%%%%%%%%%%%%%%%%%%%%%%%%
%%%%%%%%%%%%%%%%%%%%%%%%%%%
\zeta_5^4
%%%%%%%%%%%%%%%%%%
\theta^2
\left[
\begin{array}{c}
\frac15 \\
\frac95
\end{array}
\right](0,\tau)
%%%%%%%%%%%%%%%%%%%%
\theta
\left[
\begin{array}{c}
\frac15 \\
\frac75
\end{array}
\right](0,\tau)
\Bigg\}  \\
%%%%%%%%%%%%%%%%%%%%%%%%%%%%%%%%%%%%%%%%%%%%%%
%%%%%%%%%%%%%%%%%%%%%%%%%%%%%%%%%%%%%%%%%%%%%%
%%%%%%%%%%%%%%%%%%%%%%%%%%%%%%%%%%%%%%%%%%%%%%
=&
\theta
\left[
\begin{array}{c}
1 \\
\frac15
\end{array}
\right](0,\tau)
%%%%%%%%%%%%%%
\Bigg\{
\zeta_5^2
%%%%%%%%%%%%%%%%
\theta^2
\left[
\begin{array}{c}
\frac15 \\
\frac75
\end{array}
\right](0,\tau)
%%%%%%%%%%%%%%%%%%%%%
\theta
\left[
\begin{array}{c}
\frac15 \\
\frac15
\end{array}
\right](0,\tau)
%%%%%%%%%%%%%%%%%%%%%%%%%%%%%%%%%%%
%%%%%%%%%%%%%%%%%%%%%%%%%%%%%%%%%%%%
-
%%%%%%%%%%%%%%%%%%%%%%%%%%%%%%%%%%%%%
%%%%%%%%%%%%%%%%%%%%%%%%%%%%%%%%%%%%
\zeta_5^2
%%%%%%%%%%%%%%%%
\theta^2
\left[
\begin{array}{c}
\frac15 \\
\frac35
\end{array}
\right](0,\tau)
%%%%%%%%%%%%%%%%%%%%%
\theta
\left[
\begin{array}{c}
\frac15 \\
\frac95
\end{array}
\right](0,\tau)
\Bigg\}. 
\end{align*}
\par
Suppose that for every $\tau\in\mathbb{H}^2,$ 
\begin{equation*}
\zeta_5^2
%%%%%%%%%%%%%%%%
\theta^2
\left[
\begin{array}{c}
\frac15 \\
\frac75
\end{array}
\right](0,\tau)
%%%%%%%%%%%%%%%%%%%%%
\theta
\left[
\begin{array}{c}
\frac15 \\
\frac15
\end{array}
\right](0,\tau)
%%%%%%%%%%%%%%%%%%%%%%%%%%%%%%%%%%%
%%%%%%%%%%%%%%%%%%%%%%%%%%%%%%%%%%%%
-
%%%%%%%%%%%%%%%%%%%%%%%%%%%%%%%%%%%%%
%%%%%%%%%%%%%%%%%%%%%%%%%%%%%%%%%%%%
\zeta_5^2
%%%%%%%%%%%%%%%%
\theta^2
\left[
\begin{array}{c}
\frac15 \\
\frac35
\end{array}
\right](0,\tau)
%%%%%%%%%%%%%%%%%%%%%
\theta
\left[
\begin{array}{c}
\frac15 \\
\frac95
\end{array}
\right](0,\tau)
%%%%%%%%%%%%%%%%%%%%%%%%%
%%%%%%%%%%%%%%%%%%%%%%
%%%%%%%%%%%%%%%%%%%%%%%%%%%
\equiv 0,
\end{equation*}
which implies that 
\begin{equation*}
\theta^2
\left[
\begin{array}{c}
\frac15 \\
\frac15
\end{array}
\right](0,\tau)
%%%%%%%%%%%%%%
\theta
\left[
\begin{array}{c}
\frac15 \\
\frac35
\end{array}
\right](0,\tau)
%%%%%%%%%%%%%%%%%%%%%%%%%%%%
%%%%%%%%%%%%%%%%%%%%%%%%%%%%%
-
%%%%%%%%%%%%%%%%%%%%%%%%%%%%
%%%%%%%%%%%%%%%%%%%%%%%%%%%
\zeta_5^4
%%%%%%%%%%%%%%%%%%
\theta^2
\left[
\begin{array}{c}
\frac15 \\
\frac95
\end{array}
\right](0,\tau)
%%%%%%%%%%%%%%%%%%%%
\theta
\left[
\begin{array}{c}
\frac15 \\
\frac75
\end{array}
\right](0,\tau)
%%%%%%%%%%%%%%%%%%%%%%%%%%%%%%
%%%%%%%%%%%%%%%%%%%%%%%%%%%%%%%
%%%%%%%%%%%%%%%%%%%%%%%%%%%%%%%%
\equiv 0. 
\end{equation*}
Thus, we have 
\begin{equation*}
\zeta_5^2
-
\left(
\frac
{
\theta
\left[
\begin{array}{c}
\frac15 \\
\frac95
\end{array}
\right](0,\tau)
}
{
\theta
\left[
\begin{array}{c}
\frac15 \\
\frac15
\end{array}
\right](0,\tau)
}
\right)^5
\equiv0,
\end{equation*}
which implies that 
$
\theta
\left[
\begin{array}{c}
1/5 \\
9/5
\end{array}
\right](0,\tau)
/
\theta
\left[
\begin{array}{c}
1/5 \\
1/5
\end{array}
\right](0,\tau)
$ 
is a constant. 
Taking the limit $\tau \longrightarrow i \infty,$ 
we obtain 
\begin{equation*}
\theta
\left[
\begin{array}{c}
\frac15 \\
\frac95
\end{array}
\right](0,100\tau)
=
\exp
\left(
\frac{8 \pi i}{50}
\right)
%%%%%%%%%%%%%
\theta
\left[
\begin{array}{c}
\frac15 \\
\frac15
\end{array}
\right](0,100\tau),
\end{equation*}
where we used the Jacobi triple product identity (\ref{eqn:Jacobi}). 
Since 
\begin{align*}
\theta
\left[
\begin{array}{c}
\frac15 \\
\frac95
\end{array}
\right](0,100\tau)
=&
\sum_{n\in\mathbb{Z}}
\exp
(2\pi i)
\left[
\left(
n
+
\frac{1}{10}
\right)
\frac{9}{10}
\right]
x^{(10n+1)^2}, \\
%%%%%%%%%%%%%%%%%%%%%%%%%%%%
%%%%%%%%%%%%%%%%%%%%%%%%%%%%%%%%
\theta
\left[
\begin{array}{c}
\frac15 \\
\frac15
\end{array}
\right](0,100\tau)
=&
\sum_{n\in\mathbb{Z}}
\exp
(2\pi i)
\left[
\left(
n
+
\frac{1}{10}
\right)
\frac{1}{10}
\right]
x^{(10n+1)^2},  \,\,x=\exp(\pi i \tau), 
\end{align*}
comparing the coefficients of $x^{121},$ 
we obtain 
$
\zeta_5^3=1,
$ 
which is impossible. 
\par
Thus, 
it follows that 
\begin{equation*}
\zeta_5^2
%%%%%%%%%%%%%%%%
\theta^2
\left[
\begin{array}{c}
\frac15 \\
\frac75
\end{array}
\right](0,\tau)
%%%%%%%%%%%%%%%%%%%%%
\theta
\left[
\begin{array}{c}
\frac15 \\
\frac15
\end{array}
\right](0,\tau)
%%%%%%%%%%%%%%%%%%%%%%%%%%%%%%%%%%%
%%%%%%%%%%%%%%%%%%%%%%%%%%%%%%%%%%%%
-
%%%%%%%%%%%%%%%%%%%%%%%%%%%%%%%%%%%%%
%%%%%%%%%%%%%%%%%%%%%%%%%%%%%%%%%%%%
\zeta_5^2
%%%%%%%%%%%%%%%%
\theta^2
\left[
\begin{array}{c}
\frac15 \\
\frac35
\end{array}
\right](0,\tau)
%%%%%%%%%%%%%%%%%%%%%
\theta
\left[
\begin{array}{c}
\frac15 \\
\frac95
\end{array}
\right](0,\tau)
%%%%%%%%%%%%%%%%%%%%%%%%%
%%%%%%%%%%%%%%%%%%%%%%
%%%%%%%%%%%%%%%%%%%%%%%%%%%
\not\equiv 0,
\end{equation*}
which proves the theorem. 
\end{proof}

\subsection{The case for $j=3$}
In this subsection, 
we prove the following theorem:

\begin{theorem}
\label{thm:main-rational-(2)}
{\it
$
\scriptsize
{
\theta
\left[
\begin{array}{c}
1 \\
1/5
\end{array}
\right]
/
\theta
\left[
\begin{array}{c}
1 \\
3/5
\end{array}
\right]
}
$ 
\normalsize{is given by a rational expression of four of } 
$
\scriptsize
{
\theta
\left[
\begin{array}{c}
3/5 \\
k/5
\end{array}
\right], 
}
$ 
where 
$
k=1,3,5,7,9.
$
}
\end{theorem}

For the proof, 
we show the following proposition and theorem, 
which can be proved in the same way as in the previous subsection.

\begin{proposition}
\label{prop:three-theta-3/5}
{\it 
\normalsize{
For every $(\zeta,\tau)\in\mathbb{C}\times\mathbb{H}^2,$
we have 
} 
\small
{
\begin{align}
&\theta
\left[
\begin{array}{c}
1 \\
\frac35
\end{array}
\right]
%%%%%%
\theta^2 
\left[
\begin{array}{c}
\frac35 \\
\frac15
\end{array}
\right](\zeta,\tau)
%%%%%%%%%%%%%
\theta
\left[
\begin{array}{c}
\frac35 \\
\frac35
\end{array}
\right](\zeta,\tau)
%%%%%%
%%%%%%%%
+\zeta_5
\theta
\left[
\begin{array}{c}
1 \\
\frac15
\end{array}
\right]
%%%%%%
\theta^2 
\left[
\begin{array}{c}
\frac35 \\
\frac35
\end{array}
\right](\zeta,\tau)
%%%%%%%%%%%%%
\theta
\left[
\begin{array}{c}
\frac35 \\
\frac95
\end{array}
\right](\zeta,\tau)    \notag \\
%%%%%%%%%%%%%%%
%%%%%%%%%%%
-&\zeta_5^2
\theta
\left[
\begin{array}{c}
1 \\
\frac35
\end{array}
\right]
%%%%%%
\theta^2 
\left[
\begin{array}{c}
\frac35 \\
\frac95
\end{array}
\right](\zeta,\tau)
%%%%%%%%%%%%%
\theta
\left[
\begin{array}{c}
\frac35 \\
\frac75
\end{array}
\right](\zeta,\tau) 
%%%%%%%%%%%%%%%%%%%%%
%%%%%%%%%%%%%%%%%%
-\zeta_5
\theta
\left[
\begin{array}{c}
1 \\
\frac15
\end{array}
\right]
%%%%%%
\theta^2 
\left[
\begin{array}{c}
\frac35 \\
\frac75
\end{array}
\right](\zeta,\tau)
%%%%%%%%%%%%%
\theta
\left[
\begin{array}{c}
\frac35 \\
\frac15
\end{array}
\right](\zeta,\tau) =0,  \label{eqn:three-theta-3/5-5/5}
\end{align}
%%%%%%%%%%%%%%%%%%%%%%%%%%%%%%%%%%%%%%%%%%%%%%%%%%%%%%%%%%%%%%%%%%%%%%%%%%
%%%%%%%%%%%%%%%%%%%%%%%%%%%%%%%%%%%%%%%%%%%%%%%%%%%%%%%%%%%%%%%%%%%%%%%%%%%
%%%%%%%%%%%%%%%%%%%%%%%%%%%%%%%%%%%%%%%%%%%%%%%%%%%%%%%%%%%%%%%%%%%%%%%%%%%
%%%%%%%%%%%%%%%%%%%%%%%%%%%%%%%%%%%%%%%%%%%%%%%%%%%%%%%%%%%%%%%%%%%%%%%%%%
\begin{align}
&\theta
\left[
\begin{array}{c}
1 \\
\frac15
\end{array}
\right]
%%%%%%
\theta^2 
\left[
\begin{array}{c}
\frac35 \\
\frac15
\end{array}
\right](\zeta,\tau)
%%%%%%%%%%%%%
\theta
\left[
\begin{array}{c}
\frac35 \\
1
\end{array}
\right](\zeta,\tau)
%%%%%%
%%%%%%%%
-
\zeta_5
%%%%%%%%%%%%%%%
\theta
\left[
\begin{array}{c}
1 \\
\frac35
\end{array}
\right]
%%%%%%
\theta^2 
\left[
\begin{array}{c}
\frac35 \\
1
\end{array}
\right](\zeta,\tau)
%%%%%%%%%%%%%
\theta
\left[
\begin{array}{c}
\frac35 \\
\frac75
\end{array}
\right](\zeta,\tau)    \notag \\
%%%%%%%%%%%%%%%
%%%%%%%%%%%
+&
\zeta_5
%%%%%%%%%%%%%%
\theta
\left[
\begin{array}{c}
1 \\
\frac15
\end{array}
\right]
%%%%%%
\theta^2 
\left[
\begin{array}{c}
\frac35 \\
\frac75
\end{array}
\right](\zeta,\tau)
%%%%%%%%%%%%%
\theta
\left[
\begin{array}{c}
\frac35 \\
\frac35
\end{array}
\right](\zeta,\tau) 
%%%%%%%%%%%%%%%%%%%%%
%%%%%%%%%%%%%%%%%%
-
\theta
\left[
\begin{array}{c}
1 \\
\frac35
\end{array}
\right]
%%%%%%
\theta^2 
\left[
\begin{array}{c}
\frac35 \\
\frac35
\end{array}
\right](\zeta,\tau)
%%%%%%%%%%%%%
\theta
\left[
\begin{array}{c}
\frac35 \\
\frac15
\end{array}
\right](\zeta,\tau) =0,  \label{eqn:three-theta-3/5-7/5}
\end{align}
%%%%%%%%%%%%%%%%%%%%%%%%%%%%%%%%%%%%%%%%%%%%%%%%%%%%%%%%%%%%%%%%%%
%%%%%%%%%%%%%%%%%%%%%%%%%%%%%%%%%%%%%%%%%%%%%%%%%%%%%%%%%%%%%%%%%%%%%%
%%%%%%%%%%%%%%%%%%%%%%%%%%%%%%%%%%%%%%%%%%%%%%%%%%%%%%%%%%%%%%%%%%%%%
%%%%%%%%%%%%%%%%%%%%%%%%%%%%%%%%%%%%%%%%%%%%%%%%%%%%%%%%%%%%%%%%%%%%%
\begin{align}
&
\theta
\left[
\begin{array}{c}
1 \\
\frac15
\end{array}
\right]
%%%%%%
\theta^2 
\left[
\begin{array}{c}
\frac35 \\
\frac15
\end{array}
\right](\zeta,\tau)
%%%%%%%%%%%%%
\theta
\left[
\begin{array}{c}
\frac35 \\
\frac75
\end{array}
\right](\zeta,\tau)
%%%%%%
%%%%%%%%
-
\zeta_5
%%%%%%%%%%%%%%%
\theta
\left[
\begin{array}{c}
1 \\
\frac35
\end{array}
\right]
%%%%%%
\theta^2 
\left[
\begin{array}{c}
\frac35 \\
\frac75
\end{array}
\right](\zeta,\tau)
%%%%%%%%%%%%%
\theta
\left[
\begin{array}{c}
\frac35 \\
1
\end{array}
\right](\zeta,\tau)    \notag \\
%%%%%%%%%%%%%%%
%%%%%%%%%%%
+&
\zeta_5
%%%%%%%%%%%%%%
\theta
\left[
\begin{array}{c}
1 \\
\frac15
\end{array}
\right]
%%%%%%
\theta^2 
\left[
\begin{array}{c}
\frac35 \\
1
\end{array}
\right](\zeta,\tau)
%%%%%%%%%%%%%
\theta
\left[
\begin{array}{c}
\frac35 \\
\frac95
\end{array}
\right](\zeta,\tau) 
%%%%%%%%%%%%%%%%%%%%%
%%%%%%%%%%%%%%%%%%
+
\zeta_5
%%%%%%%%%%
\theta
\left[
\begin{array}{c}
1 \\
\frac35
\end{array}
\right]
%%%%%%
\theta^2 
\left[
\begin{array}{c}
\frac35 \\
\frac95
\end{array}
\right](\zeta,\tau)
%%%%%%%%%%%%%
\theta
\left[
\begin{array}{c}
\frac35 \\
\frac15
\end{array}
\right](\zeta,\tau) =0,  \label{eqn:three-theta-3/5-9/5}
\end{align}
%%%%%%%%%%%%%%%%%%%%%%%%%%%%%%%%%%%%%%%%%%%%%%%%%%%%%%%%%%%%%%%%%%%%%%%%%%%%%%%%%%%%%%
%%%%%%%%%%%%%%%%%%%%%%%%%%%%%%%%%%%%%%%%%%%%%%%%%%%%%%%%%%%%%%%%%%%%%%%%%%%%%%%%%%%%
%%%%%%%%%%%%%%%%%%%%%%%%%%%%%%%%%%%%%%%%%%%%%%%%%%%%%%%%%%%%%%%%%%%%%%%%%%%%%%%%%%%%%%%
%%%%%%%%%%%%%%%%%%%%%%%%%%%%%%%%%%%%%%%%%%%%%%%%%%%%%%%%%%%%%%%%%%%%%%%%%%%%%%%%%%%%%%%%%%
\begin{align}
&
\theta
\left[
\begin{array}{c}
1 \\
\frac35
\end{array}
\right]
%%%%%%
\theta^2 
\left[
\begin{array}{c}
\frac35 \\
\frac15
\end{array}
\right](\zeta,\tau)
%%%%%%%%%%%%%
\theta
\left[
\begin{array}{c}
\frac35 \\
\frac95
\end{array}
\right](\zeta,\tau)
%%%%%%
%%%%%%%%
+
%%%%%%%%%%%%%%%%%%%%
%%%%%%%%%%%%%%%%%%%%%
\zeta_5
%%%%%%%%%%%%%%%
\theta
\left[
\begin{array}{c}
1 \\
\frac15
\end{array}
\right]
%%%%%%
\theta^2 
\left[
\begin{array}{c}
\frac35 \\
\frac95
\end{array}
\right](\zeta,\tau)
%%%%%%%%%%%%%
\theta
\left[
\begin{array}{c}
\frac35 \\
\frac35
\end{array}
\right](\zeta,\tau)    \notag \\
%%%%%%%%%%%%%%%
%%%%%%%%%%%
-&
\theta
\left[
\begin{array}{c}
1 \\
\frac35
\end{array}
\right]
%%%%%%
\theta^2 
\left[
\begin{array}{c}
\frac35 \\
\frac35
\end{array}
\right](\zeta,\tau)
%%%%%%%%%%%%%
\theta
\left[
\begin{array}{c}
\frac35 \\
1
\end{array}
\right](\zeta,\tau) 
%%%%%%%%%%%%%%%%%%%%%
%%%%%%%%%%%%%%%%%%
+
\theta
\left[
\begin{array}{c}
1 \\
\frac15
\end{array}
\right]
%%%%%%
\theta^2 
\left[
\begin{array}{c}
\frac35 \\
1
\end{array}
\right](\zeta,\tau)
%%%%%%%%%%%%%
\theta
\left[
\begin{array}{c}
\frac35 \\
\frac15
\end{array}
\right](\zeta,\tau) =0,  \label{eqn:three-theta-3/5-1/5}
\end{align}
%%%%%%%%%%%%%%%%%%%%%%%%%%%%%%%%%%%%%%%%%%%
%%%%%%%%%%%%%%%%%%%%%%%%%%%%%%%%%%%%%%%%%%%%%
%%%%%%%%%%%%%%%%%%%%%%%%%%%%%%%%%%%%%%%%%%
%%%%%%%%%%%%%%%%%%%%%%%%%%%%%%%%%%%%%%%%%%
\begin{align}
&
\theta
\left[
\begin{array}{c}
1 \\
\frac15
\end{array}
\right]
%%%%%%
\theta^2 
\left[
\begin{array}{c}
\frac35 \\
\frac35
\end{array}
\right](\zeta,\tau)
%%%%%%%%%%%%%
\theta
\left[
\begin{array}{c}
\frac35 \\
\frac75
\end{array}
\right](\zeta,\tau)
%%%%%%
%%%%%%%%
-
%%%%%%%%%%%%%%%%%%%%
%%%%%%%%%%%%%%%%%%%%%
\zeta_5
%%%%%%%%%%%%%%%
\theta
\left[
\begin{array}{c}
1 \\
\frac35
\end{array}
\right]
%%%%%%
\theta^2 
\left[
\begin{array}{c}
\frac35 \\
\frac75
\end{array}
\right](\zeta,\tau)
%%%%%%%%%%%%%
\theta
\left[
\begin{array}{c}
\frac35 \\
\frac95
\end{array}
\right](\zeta,\tau)   \notag \\
%%%%%%%%%%%%%%%
%%%%%%%%%%%
+&
\zeta_5
%%%%%%%%%%%%%%
\theta
\left[
\begin{array}{c}
1 \\
\frac15
\end{array}
\right]
%%%%%%
\theta^2 
\left[
\begin{array}{c}
\frac35 \\
\frac95
\end{array}
\right](\zeta,\tau)
%%%%%%%%%%%%%
\theta
\left[
\begin{array}{c}
\frac35 \\
1
\end{array}
\right](\zeta,\tau) 
%%%%%%%%%%%%%%%%%%%%%
%%%%%%%%%%%%%%%%%%
-
\theta
\left[
\begin{array}{c}
1 \\
\frac35
\end{array}
\right]
%%%%%%
\theta^2 
\left[
\begin{array}{c}
\frac35 \\
1
\end{array}
\right](\zeta,\tau)
%%%%%%%%%%%%%
\theta
\left[
\begin{array}{c}
\frac35 \\
\frac35
\end{array}
\right](\zeta,\tau) =0.  \label{eqn:three-theta-3/5-3/5}
\end{align}
}
}
\end{proposition}

\begin{theorem}
\label{thm:for-main(2)}
{\it
For every $\tau\in\mathbb{H}^2,$ 
we have
\begin{equation}
\label{eqn:rational-(3/5,1)delete}  
\frac
{
\theta
\left[
\begin{array}{c}
1 \\
\frac15
\end{array}
\right](0,\tau)
}
{
\theta
\left[
\begin{array}{c}
1 \\
\frac35
\end{array}
\right](0,\tau)
}
=
\frac
{
\theta^2
\left[
\begin{array}{c}
\frac35 \\
\frac15
\end{array}
\right](0,\tau)
%%%%%%%%%%%%%
\theta
\left[
\begin{array}{c}
\frac35 \\
\frac35
\end{array}
\right](0,\tau)
%%%%%%%%%%%%%%%
%%%%%%%%%%%%%%%%
-
%%%%%%%%%%%%%%%%%
%%%%%%%%%%%%%%%%%%%
\zeta_5^2
%%%%%%%%%%%%%
\theta^2
\left[
\begin{array}{c}
\frac35 \\
\frac95
\end{array}
\right](0,\tau)
%%%%%%%%%%%%
\theta
\left[
\begin{array}{c}
\frac35 \\
\frac75
\end{array}
\right](0,\tau)
}
{
\zeta_5
\theta^2
\left[
\begin{array}{c}
\frac35 \\
\frac75
\end{array}
\right](0,\tau)
%%%%%%%%%%%%%
\theta
\left[
\begin{array}{c}
\frac35 \\
\frac15
\end{array}
\right](0,\tau)
%%%%%%%%%%%%%%%
%%%%%%%%%%%%%%%%
-
%%%%%%%%%%%%%%%%%
%%%%%%%%%%%%%%%%%%%
\zeta_5
%%%%%%%%%%
\theta^2
\left[
\begin{array}{c}
\frac35 \\
\frac35
\end{array}
\right](0,\tau)
%%%%%%%%%%%%
\theta
\left[
\begin{array}{c}
\frac35 \\
\frac95
\end{array}
\right](0,\tau)
},          
\end{equation}    
%%%%%%%%%%%%%%%%%%%%%%%%%%%%%%%%%%%%%%
%%%%%%%%%%%%%%%%%%%%%%%%%%%%%%%%%%%%%%
%%%%%%%%%%%%%%%%%%%%%%%%%%%%%%%%%%%%%%%
%%%%%%%%%%%%%%%%%%%%%%%%%%%%%%%%%%%%%%%%%
%%%%%%%%%%%%%%%%%%%%%%%%%%%%%%%%%%%%%%%%%
\begin{equation}
\label{eqn:rational-(3/5,9/5)delete}
\frac
{
\theta
\left[
\begin{array}{c}
1 \\
\frac15
\end{array}
\right](0,\tau)
}
{
\theta
\left[
\begin{array}{c}
1 \\
\frac35
\end{array}
\right](0,\tau)
}
=
\frac
{
\zeta_5
%%%%%%%%%%%%
\theta^2
\left[
\begin{array}{c}
\frac35 \\
1
\end{array}
\right](0,\tau)
%%%%%%%%%%%%%
\theta
\left[
\begin{array}{c}
\frac35 \\
\frac75
\end{array}
\right](0,\tau)
%%%%%%%%%%%%%%%
%%%%%%%%%%%%%%%%
+
%%%%%%%%%%%%%%%%%
%%%%%%%%%%%%%%%%%%%
\theta^2
\left[
\begin{array}{c}
\frac35 \\
\frac35
\end{array}
\right](0,\tau)
%%%%%%%%%%%%
\theta
\left[
\begin{array}{c}
\frac35 \\
\frac15
\end{array}
\right](0,\tau)
}
{
\theta^2
\left[
\begin{array}{c}
\frac35 \\
\frac15
\end{array}
\right](0,\tau)
%%%%%%%%%%%%%
\theta
\left[
\begin{array}{c}
\frac35 \\
1
\end{array}
\right](0,\tau)
%%%%%%%%%%%%%%%
%%%%%%%%%%%%%%%%
+
%%%%%%%%%%%%%%%%%
%%%%%%%%%%%%%%%%%%%
\zeta_5
%%%%%%%%%%
\theta^2
\left[
\begin{array}{c}
\frac35 \\
\frac75
\end{array}
\right](0,\tau)
%%%%%%%%%%%%
\theta
\left[
\begin{array}{c}
\frac35 \\
\frac35
\end{array}
\right](0,\tau)
},                
\end{equation}
%%%%%%%%%%%%%%%%%%%%%%%%%%%%%%%%%%%%%%
%%%%%%%%%%%%%%%%%%%%%%%%%%%%%%%%%%%%%%
%%%%%%%%%%%%%%%%%%%%%%%%%%%%%%%%%%%%%%%
%%%%%%%%%%%%%%%%%%%%%%%%%%%%%%%%%%%%%%%%%
%%%%%%%%%%%%%%%%%%%%%%%%%%%%%%%%%%%%%%%%%
\begin{equation}
\label{eqn:rational-(3/5,3/5)delete}  
\frac
{
\theta
\left[
\begin{array}{c}
1 \\
\frac15
\end{array}
\right](0,\tau)
}
{
\theta
\left[
\begin{array}{c}
1 \\
\frac35
\end{array}
\right](0,\tau)
}
=
\frac
{
\zeta_5
%%%%%%%%%%%%
\theta^2
\left[
\begin{array}{c}
\frac35 \\
\frac75
\end{array}
\right](0,\tau)
%%%%%%%%%%%%%
\theta
\left[
\begin{array}{c}
\frac35 \\
1
\end{array}
\right](0,\tau)
%%%%%%%%%%%%%%%
%%%%%%%%%%%%%%%%
-
%%%%%%%%%%%%%%%%%
%%%%%%%%%%%%%%%%%%%
\zeta_5
%%%%%%%%%%
\theta^2
\left[
\begin{array}{c}
\frac35 \\
\frac95
\end{array}
\right](0,\tau)
%%%%%%%%%%%%
\theta
\left[
\begin{array}{c}
\frac35 \\
\frac15
\end{array}
\right](0,\tau)
}
{
\theta^2
\left[
\begin{array}{c}
\frac35 \\
\frac15
\end{array}
\right](0,\tau)
%%%%%%%%%%%%%
\theta
\left[
\begin{array}{c}
\frac35 \\
\frac75
\end{array}
\right](0,\tau)
%%%%%%%%%%%%%%%
%%%%%%%%%%%%%%%%
+
%%%%%%%%%%%%%%%%%
%%%%%%%%%%%%%%%%%%%
\zeta_5
%%%%%%%%%%
\theta^2
\left[
\begin{array}{c}
\frac35 \\
1
\end{array}
\right](0,\tau)
%%%%%%%%%%%%
\theta
\left[
\begin{array}{c}
\frac35 \\
\frac95
\end{array}
\right](0,\tau)
},
\end{equation}              
%%%%%%%%%%%%%%%%%%%%%%%%%%%%%%%%%%%%%%%
%%%%%%%%%%%%%%%%%%%%%%%%%%%%%%%%%%%%%%
%%%%%%%%%%%%%%%%%%%%%%%%%%%%%%%%%%%%%%%
%%%%%%%%%%%%%%%%%%%%%%%%%%%%%%%%%%%%%%%%%
%%%%%%%%%%%%%%%%%%%%%%%%%%%%%%%%%%%%%%%
\begin{equation}
\label{eqn:rational-(3/5,7/5)delete}  
\frac
{
\theta
\left[
\begin{array}{c}
1 \\
\frac15
\end{array}
\right](0,\tau)
}
{
\theta
\left[
\begin{array}{c}
1 \\
\frac35
\end{array}
\right](0,\tau)
}
=
\frac
{
-
\theta^2
\left[
\begin{array}{c}
\frac35 \\
\frac15
\end{array}
\right](0,\tau)
%%%%%%%%%%%%%
\theta
\left[
\begin{array}{c}
\frac35 \\
\frac95
\end{array}
\right](0,\tau)
%%%%%%%%%%%%%%%
%%%%%%%%%%%%%%%%
+
%%%%%%%%%%%%%%%%%
%%%%%%%%%%%%%%%%%%%
\theta^2
\left[
\begin{array}{c}
\frac35 \\
\frac35
\end{array}
\right](0,\tau)
%%%%%%%%%%%%
\theta
\left[
\begin{array}{c}
\frac35 \\
1
\end{array}
\right](0,\tau)
}
{
\zeta_5
%%%%%%%%%%%%%%
\theta^2
\left[
\begin{array}{c}
\frac35 \\
\frac95
\end{array}
\right](0,\tau)
%%%%%%%%%%%%%
\theta
\left[
\begin{array}{c}
\frac35 \\
\frac35
\end{array}
\right](0,\tau)
%%%%%%%%%%%%%%%
%%%%%%%%%%%%%%%%
+
%%%%%%%%%%%%%%%%%
%%%%%%%%%%%%%%%%%%%
\theta^2
\left[
\begin{array}{c}
\frac35 \\
1
\end{array}
\right](0,\tau)
%%%%%%%%%%%%
\theta
\left[
\begin{array}{c}
\frac35 \\
\frac15
\end{array}
\right](0,\tau)
},
\end{equation}    
%%%%%%%%%%%%%%%%%%%%%%%%%%%%%%%%%%%%%%
%%%%%%%%%%%%%%%%%%%%%%%%%%%%%%%%%%%%%%
%%%%%%%%%%%%%%%%%%%%%%%%%%%%%%%%%%%%%%%
%%%%%%%%%%%%%%%%%%%%%%%%%%%%%%%%%%%%%%%%%
%%%%%%%%%%%%%%%%%%%%%%%%%%%%%%%%%%%%%%%%%
\begin{equation}
\label{eqn:rational-(3/5,1/5)delete}
\frac
{
\theta
\left[
\begin{array}{c}
1 \\
\frac15
\end{array}
\right](0,\tau)
}
{
\theta
\left[
\begin{array}{c}
1 \\
\frac35
\end{array}
\right](0,\tau)
}
=
\frac
{
\zeta_5
%%%%%%%%%%
\theta^2
\left[
\begin{array}{c}
\frac35 \\
\frac75
\end{array}
\right](0,\tau)
%%%%%%%%%%%%%
\theta
\left[
\begin{array}{c}
\frac35 \\
\frac95
\end{array}
\right](0,\tau)
%%%%%%%%%%%%%%%
%%%%%%%%%%%%%%%%
+
%%%%%%%%%%%%%%%%%
%%%%%%%%%%%%%%%%%%%
\theta^2
\left[
\begin{array}{c}
\frac35 \\
1
\end{array}
\right](0,\tau)
%%%%%%%%%%%%
\theta
\left[
\begin{array}{c}
\frac35 \\
\frac35
\end{array}
\right](0,\tau)
}
{
\theta^2
\left[
\begin{array}{c}
\frac35 \\
\frac35
\end{array}
\right](0,\tau)
%%%%%%%%%%%%%
\theta
\left[
\begin{array}{c}
\frac35 \\
\frac75
\end{array}
\right](0,\tau)
%%%%%%%%%%%%%%%
%%%%%%%%%%%%%%%%
+
%%%%%%%%%%%%%%%%%
%%%%%%%%%%%%%%%%%%%
\zeta_5
%%%%%%%%%%%
\theta^2
\left[
\begin{array}{c}
\frac35 \\
\frac95
\end{array}
\right](0,\tau)
%%%%%%%%%%%%
\theta
\left[
\begin{array}{c}
\frac35 \\
1
\end{array}
\right](0,\tau)
}.            
\end{equation}
}
\end{theorem}

\section{Rational expressions of theta constants (2)}
\label{sec:rational(2)}

\begin{theorem}
\label{thm-rational-(1/5,1/5)-(3/5,3/5)}
{\it
For every $\tau\in\mathbb{H}^2,$ we have 
\begin{equation*}
\theta
\left[
\begin{array}{c}
1 \\
\frac15
\end{array}
\right]
=
\zeta_5^2
\frac
{
\theta^3
\left[
\begin{array}{c}
\frac35 \\
\frac95
\end{array}
\right]
%%%%%%
\theta
\left[
\begin{array}{c}
\frac15 \\
1
\end{array}
\right]
+
\zeta_5^2
\theta^3
\left[
\begin{array}{c}
\frac35 \\
\frac35
\end{array}
\right]
%%%%%%%%%%
\theta
\left[
\begin{array}{c}
\frac15 \\
\frac95
\end{array}
\right]
}
{
\theta^3
\left[
\begin{array}{c}
\frac15 \\
\frac75
\end{array}
\right]
}, \,\,
%%%%%%%%%%%%%%%%%%%%%%
%%%%%%%%%%%%%%%%%%%%%%%
\theta
\left[
\begin{array}{c}
1 \\
\frac35
\end{array}
\right]
=
\zeta_5
\frac
{
\theta^3
\left[
\begin{array}{c}
\frac15 \\
\frac35
\end{array}
\right]
%%%%%%
\theta
\left[
\begin{array}{c}
\frac35 \\
1
\end{array}
\right]
-
\zeta_5^2
\theta^3
\left[
\begin{array}{c}
\frac15 \\
\frac15
\end{array}
\right]
%%%%%%%%%%
\theta
\left[
\begin{array}{c}
\frac35 \\
\frac75
\end{array}
\right]
}
{
\theta^3
\left[
\begin{array}{c}
\frac35 \\
\frac15
\end{array}
\right]
}. 
\end{equation*}
}
\end{theorem}

\begin{proof}
Consider the following elliptic functions:
\begin{equation*}
\varphi(z)
=
\frac
{
\theta^3
\left[
\begin{array}{c}
1 \\
1
\end{array}
\right](z)
}
{
%%%%
\theta
\left[
\begin{array}{c}
\frac15 \\
\frac95
\end{array}
\right](z)
%%%%%%%%%%%%
\theta
\left[
\begin{array}{c}
\frac15 \\
\frac35
\end{array}
\right](z)
%%%%%%%%%%%%%%%
\theta
\left[
\begin{array}{c}
\frac35 \\
\frac35
\end{array}
\right](z)
}, 
\,\,
\psi(z)
=
\frac
{
\theta^3
\left[
\begin{array}{c}
1 \\
1
\end{array}
\right](z)
}
{
%%%%
\theta
\left[
\begin{array}{c}
\frac35 \\
\frac35
\end{array}
\right](z)
%%%%%%%%%%%%
\theta
\left[
\begin{array}{c}
\frac35 \\
\frac95
\end{array}
\right](z)
%%%%%%%%%%%%%%%
\theta
\left[
\begin{array}{c}
-\frac15 \\
\frac95
\end{array}
\right](z)
}. 
\end{equation*}
By $\varphi(z),$ 
we derive the rational expression of 
$
\theta
\left[
\begin{array}{c}
1 \\
\frac15
\end{array}
\right].
$ 
The other can be obtained by $\psi(z)$ in the same way. 
\par
We first note that 
in the fundamental parallelogram, 
the poles of $\varphi(z)$ are $z=(2\tau-2)/5,$ $(2\tau+1)/5$ and $(\tau+1)/5.$ 
The direct calculation yields 
\begin{equation*}
\mathrm{Res}\left(\varphi(z), \frac{2\tau-2}{5} \right)
=
-\zeta_5
\frac
{
\theta^3
\left[
\begin{array}{c}
\frac15 \\
\frac95
\end{array}
\right]
%%%%%%%%%%%%%%%%
}
{
\theta^{\prime}
\left[
\begin{array}{c}
1 \\
1
\end{array}
\right]
%%%%%%%%%%%%%
\theta
\left[
\begin{array}{c}
1 \\
\frac15
\end{array}
\right]
%%%%%%%%%%%%%%%
\theta
\left[
\begin{array}{c}
\frac35 \\
\frac15
\end{array}
\right]
},
\end{equation*}
%%%%%%%%%%%%%%%%%%%%%%%%%%%%%%%%%%
\begin{equation*}
\mathrm{Res}\left(\varphi(z), \frac{2\tau+1}{5} \right)
=
-
\frac
{
\theta^3
\left[
\begin{array}{c}
\frac15 \\
\frac35
\end{array}
\right]
%%%%%%%%%%%%%
}
{
\theta^{\prime}
\left[
\begin{array}{c}
1 \\
1
\end{array}
\right]
%%%%%%%%%%%%%
\theta
\left[
\begin{array}{c}
1 \\
\frac15
\end{array}
\right]
%%%%%%%%%%%%%
\theta
\left[
\begin{array}{c}
\frac35 \\
1
\end{array}
\right]
},
\end{equation*}
%%%%%%%%%%%%%%%%%%%%%%%%%%%%%%%%%%
%%%%%%%%%%%%%%%%%%%%%%%%%%%%%%%%%%%%
and
%%%%%%%%%%%%%%%%%%%%%%%%%%%%%%%
%%%%%%%%%%%%%%%%%%%%%%%%%%%%%%%%%%%
\begin{equation*}
\mathrm{Res}\left(\varphi(z), \frac{\tau+1}{5} \right)
=
\zeta_5^4
\frac
{
\theta^3
\left[
\begin{array}{c}
\frac35 \\
\frac35
\end{array}
\right]
%%%%%%%%%%%%%%%%
}
{
\theta^{\prime}
\left[
\begin{array}{c}
1 \\
1
\end{array}
\right]
\theta
\left[
\begin{array}{c}
\frac35 \\
\frac15
\end{array}
\right]
\theta
\left[
\begin{array}{c}
\frac35 \\
1
\end{array}
\right]
}.  
\end{equation*}
Since 
$
\mathrm{Res}\left(\varphi(z),(2\tau-2)/5 \right)+\mathrm{Res}\left(\varphi(z), (2\tau+1)/5 \right)+\mathrm{Res}\left(\varphi(z),(\tau+1)/5 \right)=0,
$ 
we can obtain the rational expression. 
\end{proof}

\begin{theorem}
\label{thm-rational-(1/5,3/5)-(3/5,9/5)}
{\it
For every $\tau\in\mathbb{H}^2,$ we have 
\begin{equation*}
\theta
\left[
\begin{array}{c}
1 \\
\frac15
\end{array}
\right]
=
\zeta_5
\frac
{
\zeta_5
\theta^3
\left[
\begin{array}{c}
\frac15 \\
\frac95
\end{array}
\right]
%%%%%%
\theta
\left[
\begin{array}{c}
\frac35 \\
1
\end{array}
\right]
+
\theta^3
\left[
\begin{array}{c}
\frac15 \\
\frac35
\end{array}
\right]
%%%%%%%%%%
\theta
\left[
\begin{array}{c}
\frac35 \\
\frac15
\end{array}
\right]
}
{
\theta^3
\left[
\begin{array}{c}
\frac35 \\
\frac35
\end{array}
\right]
}, \,\,
%%%%%%%%%%%%%%%%%%%%%%
%%%%%%%%%%%%%%%%%%%%%%%
\theta
\left[
\begin{array}{c}
1 \\
\frac35
\end{array}
\right]
=
\frac
{
-
\theta^3
\left[
\begin{array}{c}
\frac35 \\
\frac75
\end{array}
\right]
%%%%%%
\theta
\left[
\begin{array}{c}
\frac15 \\
1
\end{array}
\right]
+
\theta^3
\left[
\begin{array}{c}
\frac35 \\
\frac95
\end{array}
\right]
%%%%%%%%%%
\theta
\left[
\begin{array}{c}
\frac15 \\
\frac75
\end{array}
\right]
}
{
\theta^3
\left[
\begin{array}{c}
\frac15 \\
\frac15
\end{array}
\right]
}. 
\end{equation*}
}
\end{theorem}

\begin{proof}
Consider the following elliptic functions:
\begin{equation*}
\varphi(z)
=
\frac
{
\theta^3
\left[
\begin{array}{c}
1 \\
1
\end{array}
\right](z)
}
{
%%%%
\theta
\left[
\begin{array}{c}
\frac15 \\
\frac15
\end{array}
\right](z)
%%%%%%%%%%%%
\theta
\left[
\begin{array}{c}
\frac15 \\
\frac75
\end{array}
\right](z)
%%%%%%%%%%%%%%%
\theta
\left[
\begin{array}{c}
\frac35 \\
\frac75
\end{array}
\right](z)
}, 
\,\,
\psi(z)
=
\frac
{
\theta^3
\left[
\begin{array}{c}
1 \\
1
\end{array}
\right](z)
}
{
%%%%
\theta
\left[
\begin{array}{c}
\frac35 \\
\frac15
\end{array}
\right](z)
%%%%%%%%%%%%
\theta
\left[
\begin{array}{c}
\frac35 \\
\frac35
\end{array}
\right](z)
%%%%%%%%%%%%%%%
\theta
\left[
\begin{array}{c}
-\frac15 \\
\frac15
\end{array}
\right](z)
}. 
\end{equation*}
The theorem can be proved in the same way as Theorem \ref{thm-rational-(1/5,1/5)-(3/5,3/5)}.  
\end{proof}

\begin{theorem}
\label{thm-rational-(1/5,7/5)-(3/5,1/5)}
{\it
For every $\tau\in\mathbb{H}^2,$ we have 
\begin{equation*}
\theta
\left[
\begin{array}{c}
1 \\
\frac15
\end{array}
\right]
=
\zeta_5
\frac
{
\theta^3
\left[
\begin{array}{c}
\frac15 \\
\frac15
\end{array}
\right]
%%%%%%
\theta
\left[
\begin{array}{c}
\frac35 \\
1
\end{array}
\right]
-
\theta^3
\left[
\begin{array}{c}
\frac15 \\
\frac75
\end{array}
\right]
%%%%%%%%%%
\theta
\left[
\begin{array}{c}
\frac35 \\
\frac95
\end{array}
\right]
}
{
\theta^3
\left[
\begin{array}{c}
\frac35 \\
\frac75
\end{array}
\right]
}, \,\,
%%%%%%%%%%%%%%%%%%%%%%
%%%%%%%%%%%%%%%%%%%%%%%
\theta
\left[
\begin{array}{c}
1 \\
\frac35
\end{array}
\right]
=
\zeta_5^4
\frac
{
\theta^3
\left[
\begin{array}{c}
\frac35 \\
\frac35
\end{array}
\right]
%%%%%%
\theta
\left[
\begin{array}{c}
\frac15 \\
1
\end{array}
\right]
-
\theta^3
\left[
\begin{array}{c}
\frac35 \\
\frac15
\end{array}
\right]
%%%%%%%%%%
\theta
\left[
\begin{array}{c}
\frac15 \\
\frac35
\end{array}
\right]
}
{
\theta^3
\left[
\begin{array}{c}
\frac15 \\
\frac95
\end{array}
\right]
}. 
\end{equation*}
}
\end{theorem}

\begin{proof}
Consider the following elliptic functions:
\begin{equation*}
\varphi(z)
=
\frac
{
\theta^3
\left[
\begin{array}{c}
1 \\
1
\end{array}
\right](z)
}
{
%%%%
\theta
\left[
\begin{array}{c}
\frac15 \\
-\frac35
\end{array}
\right](z)
%%%%%%%%%%%%
\theta
\left[
\begin{array}{c}
\frac15 \\
\frac95
\end{array}
\right](z)
%%%%%%%%%%%%%%%
\theta
\left[
\begin{array}{c}
\frac35 \\
\frac95
\end{array}
\right](z)
}, 
\,\,
\psi(z)
=
\frac
{
\theta^3
\left[
\begin{array}{c}
1 \\
1
\end{array}
\right](z)
}
{
%%%%
\theta
\left[
\begin{array}{c}
\frac35 \\
\frac15
\end{array}
\right](z)
%%%%%%%%%%%%
\theta
\left[
\begin{array}{c}
\frac35 \\
\frac75
\end{array}
\right](z)
%%%%%%%%%%%%%%%
\theta
\left[
\begin{array}{c}
-\frac15 \\
\frac75
\end{array}
\right](z)
}. 
\end{equation*}
The theorem can be proved in the same way as Theorem \ref{thm-rational-(1/5,1/5)-(3/5,3/5)}.  
\end{proof}

\begin{theorem}
\label{thm-rational-(1/5,9/5)-(3/5,7/5)}
{\it
For every $\tau\in\mathbb{H}^2,$ we have 
\begin{equation*}
\theta
\left[
\begin{array}{c}
1 \\
\frac15
\end{array}
\right]
=
-
\zeta_5^2
\frac
{
\theta^3
\left[
\begin{array}{c}
\frac35 \\
\frac15
\end{array}
\right]
%%%%%%
\theta
\left[
\begin{array}{c}
\frac15 \\
1
\end{array}
\right]
+
\zeta_5^3
\theta^3
\left[
\begin{array}{c}
\frac35 \\
\frac75
\end{array}
\right]
%%%%%%%%%%
\theta
\left[
\begin{array}{c}
\frac15 \\
\frac15
\end{array}
\right]
}
{
\theta^3
\left[
\begin{array}{c}
\frac15 \\
\frac35
\end{array}
\right]
}, 
%%%%%%%%%%%%%%%%%%%%%%
%%%%%%%%%%%%%%%%%%%%%%%
\theta
\left[
\begin{array}{c}
1 \\
\frac35
\end{array}
\right]
=
-
\zeta_5^3
\frac
{
\theta^3
\left[
\begin{array}{c}
\frac15 \\
\frac75
\end{array}
\right]
%%%%%%
\theta
\left[
\begin{array}{c}
\frac35 \\
1
\end{array}
\right]
-
\theta^3
\left[
\begin{array}{c}
\frac15 \\
\frac95
\end{array}
\right]
%%%%%%%%%%
\theta
\left[
\begin{array}{c}
\frac35 \\
\frac35
\end{array}
\right]
}
{
\theta^3
\left[
\begin{array}{c}
\frac35 \\
\frac95
\end{array}
\right]
}. 
\end{equation*}
}
\end{theorem}

\begin{proof}
Consider the following elliptic functions:
\begin{equation*}
\varphi(z)
=
\frac
{
\theta^3
\left[
\begin{array}{c}
1 \\
1
\end{array}
\right](z)
}
{
%%%%
\theta
\left[
\begin{array}{c}
\frac15 \\
-\frac35
\end{array}
\right](z)
%%%%%%%%%%%%
\theta
\left[
\begin{array}{c}
\frac15 \\
\frac95
\end{array}
\right](z)
%%%%%%%%%%%%%%%
\theta
\left[
\begin{array}{c}
\frac35 \\
\frac95
\end{array}
\right](z)
}, 
\,\,
\psi(z)
=
\frac
{
\theta^3
\left[
\begin{array}{c}
1 \\
1
\end{array}
\right](z)
}
{
%%%%
\theta
\left[
\begin{array}{c}
\frac35 \\
\frac15
\end{array}
\right](z)
%%%%%%%%%%%%
\theta
\left[
\begin{array}{c}
\frac35 \\
\frac75
\end{array}
\right](z)
%%%%%%%%%%%%%%%
\theta
\left[
\begin{array}{c}
-\frac15 \\
\frac75
\end{array}
\right](z)
}. 
\end{equation*}
The theorem can be proved in the same way as Theorem \ref{thm-rational-(1/5,1/5)-(3/5,3/5)}.  
\end{proof}

\section{Rational expressions of theta constants (3)}
\label{sec:rational(3)}

\subsection{Review of the resultant}
Set 
\begin{align*}
f(x)&=a_0x^n+a_1x^{n-1}+\cdots +a_n,  \\
g(x)&=b_0x^n+b_1x^{m-1}+\cdots +b_{m}.
\end{align*}   
We then define the resultant $R(f,g)$ by the following 
$(m+n)\times(m+n)$ determinant:
\begin{equation*}
R(f,g)=
%
.
\end{equation*}
%%%%%%%%%%%%%%%%
%%%%%%%%%%%%%%%%
\par
$R(f,g)$ shows us the condition for $f(x)$ and $g(x)$ 
to have a common zero. 
\begin{proposition}
{\it
$f(x)$ and $g(x)$ both have a common zero 
if and only if $R(f,g)=0.$ 
}
\end{proposition}

Especially, 
when 
$$
f(x)=a_0x^2+a_1x+a_2, \,\,\mathrm{and} \,\,g(x)=b_0x^2+b_1x+b_2, 
$$
the resultant is given by 
\begin{align*}
R(f,g)&=
%
 \\
&=
(a_0b_2-a_2b_0)^2-
(a_0b_1-a_1b_0)(a_1b_2-a_2b_1).
\end{align*}

\subsection{The case for $j=1$}

In this subsection, 
we prove the following theorem:

\begin{theorem}
\label{thm:main-rational-(3)}
{\it
One of 
$
\scriptsize
{
\theta
\left[
\begin{array}{c}
1/5 \\
k/5
\end{array}
\right], 
}
\,
\normalsize{
(k=1,3,5,7,9),
}
$ 
is given by two rational expressions of the other four theta constants. 
}
\end{theorem}

For the proof, 
we obtain the following proposition and theorem:

\begin{proposition}
\label{propo:two-theta-1/5}
{\it
For every $(\zeta,\tau)\in\mathbb{C}\times\mathbb{H}^2,$ 
we have 
\begin{align}
&
\theta^2
\left[
\begin{array}{c}
1 \\
\frac35
\end{array}
\right]
%%%%%%
\theta
\left[
\begin{array}{c}
\frac15 \\
\frac15
\end{array}
\right](\zeta,\tau)
%%%%%%%%%%%%%
\theta
\left[
\begin{array}{c}
\frac15 \\
\frac95
\end{array}
\right](\zeta,\tau)
%%%%%%
%%%%%%%%
-
\theta^2
\left[
\begin{array}{c}
1 \\
\frac15
\end{array}
\right]
%%%%%%
\theta 
\left[
\begin{array}{c}
\frac15 \\
\frac35
\end{array}
\right](\zeta,\tau)
%%%%%%%%%%%%%
\theta
\left[
\begin{array}{c}
\frac15 \\
\frac75
\end{array}
\right](\zeta,\tau)      \notag \\
%%%%%%%%%%%%%%%%%%%%%%%
%%%%%%%%%%%%%%%%%%%%%%%%
&
\hspace{30mm}
+\theta
\left[
\begin{array}{c}
1 \\
\frac15
\end{array}
\right]
%%%%%%
\theta
\left[
\begin{array}{c}
1 \\
\frac35
\end{array}
\right]
%%%%%%%%%%%
\theta^2
\left[
\begin{array}{c}
\frac15 \\
1
\end{array}
\right](\zeta,\tau)=0,  \label{eqn:two-theta--1/5-sum-10/5}
\end{align}
%%%%%%%%%%%%%%%%%%%%%%%%%%%%%
%%%%%%%%%%%%%%%%%%%%%%%%%%%%%
%%%%%%%%%%%%%%%%%%%%%%%%%%%%%%
%%%%%%%%%%%%%%%%%%%%%%%%%%%%%
\begin{align}
&\zeta_5^2\theta^2
\left[
\begin{array}{c}
1 \\
\frac15
\end{array}
\right]
%%%%%%
\theta
\left[
\begin{array}{c}
\frac15 \\
\frac35
\end{array}
\right](\zeta,\tau)
%%%%%%%%%%%%%
\theta
\left[
\begin{array}{c}
\frac15 \\
\frac95
\end{array}
\right](\zeta,\tau)
%%%%%%
%%%%%%%%
-
\zeta_5^2
\theta^2
\left[
\begin{array}{c}
1 \\
\frac35
\end{array}
\right]
%%%%%%
\theta 
\left[
\begin{array}{c}
\frac15 \\
1
\end{array}
\right](\zeta,\tau)
%%%%%%%%%%%%%
\theta
\left[
\begin{array}{c}
\frac15 \\
\frac75
\end{array}
\right](\zeta,\tau)    \notag \\
%%%%%%%%%%%%%%%%%%%%%%%
%%%%%%%%%%%%%%%%%%%%%%%%
&
\hspace{30mm}
+\theta
\left[
\begin{array}{c}
1 \\
\frac15
\end{array}
\right]
%%%%%%
\theta
\left[
\begin{array}{c}
1 \\
\frac35
\end{array}
\right]
%%%%%%%%%%%
\theta^2
\left[
\begin{array}{c}
\frac15 \\
\frac15
\end{array}
\right](\zeta,\tau)=0,  \label{eqn:two-theta--1/5-sum-2/5}
\end{align}
%%%%%%%%%%%%%%%%%%%%%%%%%%%%%%%
%%%%%%%%%%%%%%%%%%%%%%%%%%%%%%
%%%%%%%%%%%%%%%%%%%%%%%%%%%%%
%%%%%%%%%%%%%%%%%%%%%%%%%%%%%%%
\begin{align}
&\theta^2
\left[
\begin{array}{c}
1 \\
\frac35
\end{array}
\right]
%%%%%%
\theta
\left[
\begin{array}{c}
\frac15 \\
\frac15
\end{array}
\right](\zeta,\tau)
%%%%%%%%%%%%%
\theta
\left[
\begin{array}{c}
\frac15 \\
\frac35
\end{array}
\right](\zeta,\tau)
%%%%%%
%%%%%%%%
+
\zeta_5^2
\theta^2
\left[
\begin{array}{c}
1 \\
\frac15
\end{array}
\right]
%%%%%%
\theta 
\left[
\begin{array}{c}
\frac15 \\
1
\end{array}
\right](\zeta,\tau)
%%%%%%%%%%%%%
\theta
\left[
\begin{array}{c}
\frac15 \\
\frac95
\end{array}
\right](\zeta,\tau)    \notag \\
%%%%%%%%%%%%%%%%%%%%%%%
%%%%%%%%%%%%%%%%%%%%%%%%
&
\hspace{30mm}
-\zeta_5^2\theta
\left[
\begin{array}{c}
1 \\
\frac15
\end{array}
\right]
%%%%%%
\theta
\left[
\begin{array}{c}
1 \\
\frac35
\end{array}
\right]
%%%%%%%%%%%
\theta^2
\left[
\begin{array}{c}
\frac15 \\
\frac75
\end{array}
\right](\zeta,\tau)=0,  \label{eqn:two-theta--1/5-sum-4/5}
\end{align}
%%%%%%%%%%%%%%%%%%%%%%%%%%%%%%%%%%%%%%%
%%%%%%%%%%%%%%%%%%%%%%%%%%%%%%%%%%%%%
%%%%%%%%%%%%%%%%%%%%%%%%%%%%%%%%%%%%%%%%%
%%%%%%%%%%%%%%%%%%%%%%%%%%%%%%%%%%%%%%%%%%
\begin{align}
&\theta^2
\left[
\begin{array}{c}
1 \\
\frac15
\end{array}
\right]
%%%%%%
\theta
\left[
\begin{array}{c}
\frac15 \\
\frac15
\end{array}
\right](\zeta,\tau)
%%%%%%%%%%%%%
\theta
\left[
\begin{array}{c}
\frac15 \\
1
\end{array}
\right](\zeta,\tau)
%%%%%%
%%%%%%%%
+
\zeta_5^2
\theta^2
\left[
\begin{array}{c}
1 \\
\frac35
\end{array}
\right]
%%%%%%
\theta 
\left[
\begin{array}{c}
\frac15 \\
\frac75
\end{array}
\right](\zeta,\tau)
%%%%%%%%%%%%%
\theta
\left[
\begin{array}{c}
\frac15 \\
\frac95
\end{array}
\right](\zeta,\tau)    \notag \\
%%%%%%%%%%%%%%%%%%%%%%%
%%%%%%%%%%%%%%%%%%%%%%%%
&
\hspace{30mm}
-\theta
\left[
\begin{array}{c}
1 \\
\frac15
\end{array}
\right]
%%%%%%
\theta
\left[
\begin{array}{c}
1 \\
\frac35
\end{array}
\right]
%%%%%%%%%%%
\theta^2
\left[
\begin{array}{c}
\frac15 \\
\frac35
\end{array}
\right](\zeta,\tau)=0,  \label{eqn:two-theta--1/5-sum-6/5}
\end{align}
%%%%%%%%%%%%%%%%%%%%%%%%%%%%%%
%%%%%%%%%%%%%%%%%%%%%%%%%%%%
%%%%%%%%%%%%%%%%%%%%%%%%%%%%%%%
%%%%%%%%%%%%%%%%%%%%%%%%%%%%%%%%%
\begin{align}
&\theta^2
\left[
\begin{array}{c}
1 \\
\frac15
\end{array}
\right]
%%%%%%
\theta
\left[
\begin{array}{c}
\frac15 \\
\frac15
\end{array}
\right](\zeta,\tau)
%%%%%%%%%%%%%
\theta
\left[
\begin{array}{c}
\frac15 \\
\frac75
\end{array}
\right](\zeta,\tau)
%%%%%%
%%%%%%%%
-
\theta^2
\left[
\begin{array}{c}
1 \\
\frac35
\end{array}
\right]
%%%%%%
\theta 
\left[
\begin{array}{c}
\frac15 \\
\frac35
\end{array}
\right](\zeta,\tau)
%%%%%%%%%%%%%
\theta
\left[
\begin{array}{c}
\frac15 \\
1
\end{array}
\right](\zeta,\tau)    \notag \\
%%%%%%%%%%%%%%%%%%%%%%%
%%%%%%%%%%%%%%%%%%%%%%%%
&
\hspace{30mm}
+\zeta_5^2
\theta
\left[
\begin{array}{c}
1 \\
\frac15
\end{array}
\right]
%%%%%%
\theta
\left[
\begin{array}{c}
1 \\
\frac35
\end{array}
\right]
%%%%%%%%%%%
\theta^2
\left[
\begin{array}{c}
\frac15 \\
\frac95
\end{array}
\right](\zeta,\tau)=0.  \label{eqn:two-theta--1/5-sum-8/5}
\end{align}
}
\end{proposition}

\begin{proof}
Equation (\ref{eqn:two-theta--1/5-sum-10/5}) 
was proved by Matsuda \cite{Matsuda}, 
who proved it in the same way as Proposition \ref{prop:three-theta-1/5}. 
The others can be also proved in the same way. 
\end{proof}

\begin{theorem}
\label{thm:for-main (3)}
{\it
For every $\tau\in\mathbb{H}^2,$ we have 
\scriptsize{
\begin{equation}
\label{eqn:(1/5,1/5)-rational-(1)}  
\theta 
\left[
\begin{array}{c}
\frac15 \\
\frac15
\end{array}
\right]
=
\frac
{
\Bigg\{
\theta
\left[
\begin{array}{c}
\frac15 \\
\frac35
\end{array}
\right]
%%%%%%%%%%%%%
\theta^3
\left[
\begin{array}{c}
\frac15 \\
\frac75
\end{array}
\right]
%%%%%%%%%%%%
%%%%%%%%%%%%
-
\theta^3 
\left[
\begin{array}{c}
\frac15 \\
1
\end{array}
\right]
%%%%%%%%%%%%%
\theta
\left[
\begin{array}{c}
\frac15 \\
\frac95
\end{array}
\right]
\Bigg\}
%%%%%%%%%%%%%%%%%%
%%%%%%%%%%%%%%%%%%%
\Bigg\{
\theta^2
\left[
\begin{array}{c}
\frac15 \\
\frac75
\end{array}
\right]
%%%%%%%%%%%%%
\theta
\left[
\begin{array}{c}
\frac15 \\
\frac95
\end{array}
\right]
%%%%%%%%%%%%
%%%%%%%%%%%%
+
\zeta_5^3
\theta 
\left[
\begin{array}{c}
\frac15 \\
\frac35
\end{array}
\right]
%%%%%%%%%%%%%
\theta^2
\left[
\begin{array}{c}
\frac15 \\
1
\end{array}
\right]
\Bigg\}
}
{
\Bigg\{
\zeta_5^3
\theta^2 
\left[
\begin{array}{c}
\frac15 \\
\frac35
\end{array}
\right]
%%%%%%%%%%%%%
\theta
\left[
\begin{array}{c}
\frac15 \\
\frac75
\end{array}
\right]
%%%%%%%%%%%%%%%%%
%%%%%%%%%%%%%%%%
+
\theta 
\left[
\begin{array}{c}
\frac15 \\
1
\end{array}
\right]
%%%%%%%%%%%%%
\theta^2
\left[
\begin{array}{c}
\frac15 \\
\frac95
\end{array}
\right]
\Bigg\}^2
},
\end{equation}              
%%%%%%%%%%%%%%%%%%%%%%%%%%%%%%%
%%%%%%%%%%%%%%%%%%%%%%%%%%%%%%
%%%%%%%%%%%%%%%%%%%%%%%%%%%%%%%
%%%%%%%%%%%%%%%%%%%%%%%%%%%%%%
%%%%%%%%%%%%%%%%%%%%%%%%%%%%%%%%
%%%%%%%%%%%%%%%%%%%%%%%%%%%%%%%%%
\begin{equation}
\label{eqn:(1/5,1/5)-rational-(2)}  
\theta
\left[
\begin{array}{c}
\frac15   \\
\frac15
\end{array}
\right]
=
\frac
{
\Bigg\{
\zeta_5^2
%%%%%%%%%%%%%%%
\theta
\left[
\begin{array}{c}
\frac15   \\
1
\end{array}
\right]
%%%%%%%%%%%%%%%%%%%%%%
\theta^2
\left[
\begin{array}{c}
\frac15   \\
\frac95
\end{array}
\right]
%%%%%%%%%%%%%%%%%%%%%%%%%%%%%%%%%%%%%%%
%%%%%%%%%%%%%%%%%%%%%%%%%%%%%%%%%%%%%%%%%%
+
%%%%%%%%%%%%%%%%%%%%%%%%%%%%%%%%%%%%%%%%%
%%%%%%%%%%%%%%%%%%%%%%%%%%%%%%%%%%%%%%%%%
\theta^2
\left[
\begin{array}{c}
\frac15   \\
\frac35
\end{array}
\right]
%%%%%%%%%%%%%%%%%%%%%%%%%%%%
\theta
\left[
\begin{array}{c}
\frac15   \\
\frac75
\end{array}
\right]
\Bigg\}
%%%%%%%%%%%%%%%%%%%%%%%%%%%%%%%%%%%
%%%%%%%%%%%%%%%%%%%%%%%%%%%%%%%%%%%%
%%%%%%%%%%%%%%%%%%%%%%%%%%%%%%%%%%%%%
\Bigg\{
\theta^3
\left[
\begin{array}{c}
\frac15   \\
\frac35
\end{array}
\right]
%%%%%%%%%%%%%%%%%%%%%%%%%
\theta
\left[
\begin{array}{c}
\frac15   \\
1
\end{array}
\right]
%%%%%%%%%%%%%%%%%%%%%%%%%%%%%%%%%%%%%
%%%%%%%%%%%%%%%%%%%%%%%%%%%%%%%%%%%%%
-
%%%%%%%%%%%%%%%%%%%%%%%%%%%%%%%%%%%%%%%%
%%%%%%%%%%%%%%%%%%%%%%%%%%%%%%%%%%%%%%%%%%
\zeta_5^4
%%%%%%%%%%%%%%%%%%%%%%%%%%%%%%
\theta
\left[
\begin{array}{c}
\frac15   \\
\frac75
\end{array}
\right]
%%%%%%%%%%%%%%%%%%%%%%%%%%%%%%%
\theta^3
\left[
\begin{array}{c}
\frac15   \\
\frac95
\end{array}
\right]
\Bigg\}
}
{
\Bigg\{
\theta
\left[
\begin{array}{c}
\frac15   \\
\frac35
\end{array}
\right]
%%%%%%%%%%%%%%%%%%%%%%%%%
\theta^2
\left[
\begin{array}{c}
\frac15   \\
1
\end{array}
\right]
%%%%%%%%%%%%%%%%%%%%%%%%%%%%%%%%%%
%%%%%%%%%%%%%%%%%%%%%%%%%%%%%%%%%%%
+
%%%%%%%%%%%%%%%%%%%%%%%%%%%%%%%%%%%%
%%%%%%%%%%%%%%%%%%%%%%%%%%%%%%%%%%%%%%%
\zeta_5^2
%%%%%%%%%%%%%%%%%%%%%
\theta^2
\left[
\begin{array}{c}
\frac15   \\
\frac75
\end{array}
\right]
%%%%%%%%%%%%%%%%%%%%%%
\theta
\left[
\begin{array}{c}
\frac15   \\
\frac95
\end{array}%
\right]
\Bigg\}^2
}.
\end{equation}
}
}
\end{theorem}

\begin{proof}
By equations (\ref{eqn:two-theta--1/5-sum-10/5}) and (\ref {eqn:two-theta--1/5-sum-4/5}), 
we prove equation (\ref{eqn:(1/5,1/5)-rational-(1)}). 
The others can be proved by two equations of Proposition \ref{propo:two-theta-1/5} in the same way. 
\par
Equations (\ref{eqn:two-theta--1/5-sum-10/5}) and (\ref {eqn:two-theta--1/5-sum-4/5}) 
show that for arbitrary complex numbers $z,w\in\mathbb{C},$ 
\begin{align*}
&\theta^2
\left[
\begin{array}{c}
1 \\
\frac35
\end{array}%
\right](w,\tau)=0,  
\end{align*}
which implies that 
$ 
\theta
\left[
\begin{array}{c}
1 \\
\frac15
\end{array}%
\right]
$ satisfies the following two algebraic equations:
\begin{align*}
f(x)&=
x^2
\theta 
\left[
%
\right](w)=0.
\end{align*}
\par
Since $f(x)$ and $g(x)$ both have the same solution, 
it follows that 
{\footnotesize
\begin{align*}
R(f,g)&=
%
 \\
&=0,
\end{align*}
}
\normalsize{
which shows that 
}
\begin{align*}
&
\Bigg\{
\zeta_5^3
\theta
\left[
%
\right]
\Bigg\}.
\end{align*}
\par
Let us show that 
$$
\zeta_5^3
%%%%%%%%%%%%%%%
\theta^2 
\left[
%
\right]
\not\equiv 0.
$$
For this purpose, 
suppose that for every $\tau\in\mathbb{H}^2,$ 
$$
\zeta_5^3
%%%%%%%%%%%%%%%
\theta^2 
\left[
%
\right](0,\tau)
\equiv 0,
$$
which implies that 
\begin{equation*}
\label{eqn:case(1)}
\theta
\left[
%
\right](0,\tau)
\equiv 0.
\end{equation*}
Thus, 
it follows that 
$$
\zeta_5
+
\left(
\frac
{
\theta
\left[
%
\right](0,\tau)
}
\right)^5
\equiv 0,
$$
which implies that 
$
\theta
\left[
%
\right](0,\tau)
$ 
is a constant. 
Taking the limit $\tau\longrightarrow i \infty,$ 
we have 
$$
\theta
\left[
%
\right](0,100\tau),
$$
in which we used Jacobi's triple product identity (\ref{eqn:Jacobi}). 
Since
$$
\theta
\left[
%
\right](0,100\tau)=\sum_{n\in\mathbb{Z}}
\exp(2\pi i)
\left[
\left(n+\frac{1}{10}\right)\frac{3}{10}
\right]
x^{(10n+1)^2}, \,\,\, x=\exp(\pi i \tau),
$$
comparing the coefficients of $x^{121},$ 
we have $\zeta_5^3=1,$ which is impossible. 
\par
Therefore, 
we have 
\begin{equation*}
\theta 
\left[
%
\right]
\Bigg\}^2
}. 
\end{equation*}
\end{proof}

\subsection{The case for $j=3$}

\begin{theorem}
\label{thm:main-rational-(4)}
{\it
One of 
$
\scriptsize
{
\theta
\left[
%
\right], 
}
\,
\normalsize{
(k=1,3,5,7,9),
}
$ 
is given by two rational expressions of the other four theta constants. 
}
\end{theorem}

For the proof, we obtain the following proposition and theorem, 
which can be proved in the same way as in the previous subsection.

\begin{proposition}
\label{propo:two-theta-3/5}
{\it
For every $(\zeta,\tau)\in\mathbb{C}\times\mathbb{H}^2,$ 
we have 
\begin{align}
&\theta^2
\left[
%
\right]
\Bigg\}^2 
}.                        
\end{equation}
}
}
\end{theorem}

\section{Certain theta constant identities}  
\label{sec:certain-theta}
In this section, 
for $j=1,3,$ 
we investigate the algebraic structure of the following five theta constants: 
$$
\theta
\left[
%
\right], \,\,
(k=1,3,5,7,9).
$$ 

\subsection{The case for $j=1$  }
In this subsection, 
we prove the following theorem:

\begin{theorem}
\label{thm:alge-structure-(1)}
{\it
Among the following five theta constants, 
 
$$
\theta
\left[
%
\right], 
\,\,
(k=1,3,5,7,9),
$$ 
there exist at most three theta constants which 
are algebraically independent over $\mathbb{C}$. 
}
\end{theorem}

For the proof, 
we obtain the following theorem:

\begin{theorem}
\label{thm:for-main-(5)}
{\it
For every $\tau \in\mathbb{H}^2,$ 
we have 
%\scriptsize{
\begin{align*}
&
\Bigg\{
\theta
\left[
%
\right]
\Bigg\}.
\end{align*}
%}
}
\end{theorem}

\begin{proof}
The first equation follows from equations (\ref{eqn:(1/5,1/5)-rational-(1)}) and (\ref{eqn:(1/5,1/5)-rational-(2)}).  
The second equation is obtained by equations (\ref{eqn:(1/5,3/5)-rational-(1)}) and (\ref{eqn:(1/5,3/5)-rational-(2)}). 
The third equation is derived from equations (\ref{eqn:(1/5,1)-rational-(1)}) and (\ref{eqn:(1/5,1)-rational-(2)}). 
The fourth equation follows from equations (\ref{eqn:(1/5,7/5)-rational-(1)}) and (\ref{eqn:(1/5,7/5)-rational-(2)}). 
The fifth equation follows from equations (\ref{eqn:(1/5,9/5)-rational-(1)}) and (\ref{eqn:(1/5,9/5)-rational-(2)}). 
\end{proof}

\subsection{The case for $j=3$}
In this subsection, 
we prove the following theorem:

\begin{theorem}
\label{thm:alge-structure-(2)}
{\it

Among the following five theta constants, 
$$
\theta
\left[

\right], 
\,\,
(k=1,3,5,7,9),
$$ 
there exist at most three theta constants which  
are algebraically independent over $\mathbb{C}$. 
}
\end{theorem}

For the proof, 
we obtain the following theorem, 
which can be proved in the same way as in the previous subsection.

\begin{theorem}
{\it
For every $\tau\in\mathbb{H}^2,$ 
we have
%\scriptsize{
\begin{align*}
&
\Bigg\{
\theta
\left[
%
\right]
\Bigg\}. 
\end{align*}
%}
}
\end{theorem}

\section{The algebraic structure of modular forms of level five }  
\label{sec:algebraic-structure}
In this section, we investigate the algebraic structure 
of the following theta constants:
$$
\theta
\left[
%
\right],  \,\,
(j=1,3, k=1,3,5,7,9). 
$$

\subsection{Certain theta constant identities}

\begin{theorem}
\label{thm:for-main-(7)}
{\it
For every $\tau\in\mathbb{H}^2,$ we have 
\small
{
\begin{align}
&
\Bigg\{
\theta^2
\left[
%
\right]
\Bigg\}.    \label{eqn:(1/5,1/5)-(3/5,1/5)-delete}
\end{align}
}
}
\end{theorem}

\begin{proof}
Equation (\ref{eqn:(1/5,1)-(3/5,1)-delete}) 
follows from equations (\ref{eqn:rational-(1/5,1)delete}) 
and (\ref{eqn:rational-(3/5,1)delete}).  
%%%%%%%%%%%%%%%%%%%%%%%%%%%%%%%%%%%%
Equation (\ref{eqn:(1/5,9/5)-(3/5,9/5)-delete}) 
is obtained by equations (\ref{eqn:rational-(1/5,9/5)delete}) 
and (\ref{eqn:rational-(3/5,9/5)delete}).  
%%%%%%%%%%%%%%%%%%%%%%%%%%%%%%%%%%%%%
Equation (\ref{eqn:(1/5,3/5)-(3/5,3/5)-delete}) 
is derived from equations (\ref{eqn:rational-(1/5,3/5)delete}) 
and (\ref{eqn:rational-(3/5,3/5)delete}).  
%%%%%%%%%%%%%%%%%%%%%%%%%%%%%%%%%%%%%
Equation (\ref{eqn:(1/5,7/5)-(3/5,7/5)-delete}) 
follows from equations (\ref{eqn:rational-(1/5,7/5)delete}) 
and (\ref{eqn:rational-(3/5,7/5)delete}).  
%%%%%%%%%%%%%%%%%%%%%%%%%%%%%%%%%%%%%
Equation (\ref{eqn:(1/5,1/5)-(3/5,1/5)-delete}) 
is obtained from equations (\ref{eqn:rational-(1/5,1/5)delete}) 
and (\ref{eqn:rational-(3/5,1/5)delete}).  
\end{proof}

\subsection{Proof of Theorem \ref{thm:main-(0)}}

\begin{theorem}
\label{thm:main-(0)}
{\it
Among the following twelve theta constants, 
$$
\theta
\left[
\begin{array}{c}
\frac{j}{5} \\
\frac{k}{5}
\end{array}
\right], \,
\theta
\left[
\begin{array}{c}
1 \\
\frac{j}{5}
\end{array}
\right], 
\,\,
(j=1,3, \,k=1,3,5,7,9),
$$ 
there exist at most five theta constants which are algebraically independent over $\mathbb{C}.$
}
\end{theorem}

\begin{proof}
Theorem \ref{thm:main-(0)} follows from 
Theorems 
\ref{thm-rational-(1/5,1/5)-(3/5,3/5)},
\ref{thm:main-rational-(3)}, \ref{thm:main-rational-(4)}, \ref{thm:alge-structure-(1)}, \ref{thm:alge-structure-(2)} 
and \ref{thm:for-main-(7)}. 
\end{proof}

\end{document}